\documentclass{amsart}
\usepackage{amsmath}
\usepackage{paralist}
\usepackage{amsfonts}
\usepackage{amssymb}
\usepackage{amsthm}
\usepackage{amscd}
\usepackage{float}
\usepackage{tikz}
\usepackage{graphicx}
\usepackage[colorlinks=true]{hyperref}
\hypersetup{urlcolor=blue, citecolor=red}
\usepackage{hyperref}
\usepackage{cite}
\usepackage{url}

\newtheorem{theorem}{Theorem}[section]

\newtheorem{lemma}[theorem]{Lemma}
\newtheorem{proposition}[theorem]{Proposition}

\theoremstyle{definition}
\newtheorem{definition}[theorem]{Definition}
\newtheorem{remark}[theorem]{Remark}

\newcommand{\R}{\mathbb{R}}
\newcommand{\Z}{\mathbb{Z}}
\newcommand{\N}{\mathbb{N}}

\newcommand {\Rd}{\mathbb{R}^{d}}
\newcommand {\Rtwo}{\mathbb{R}^{2}}
\newcommand {\Hp}{\mathcal{H}^{p}_{FIO}(\Rd)}
\newcommand {\Hps}{\mathcal{H}^{s,p}_{FIO}(\Rd)}  
\newcommand {\w}{\omega}
\newcommand {\ph}{\varphi}
\newcommand{\Sd}{\mathbb{S}^{d-1}}
\newcommand{\Su}{\mathbb{S}}
\newcommand {\HT}{\mathcal{H}}
\newcommand {\Bspq}{B^{s}_{p,r}(\Rd)}
\newcommand {\Ba}{\mathcal{B}}
\newcommand {\Baspqr}{\mathcal{B}^{s}_{p,q,r}(\Rd)}
\newcommand {\Baspqq}{\mathcal{B}^{s}_{p,q,q}(\Rd)}
\newcommand {\Basppp}{\mathcal{B}^{s}_{p,p,p}(\Rd)}
\newcommand {\supp}{\mathrm{supp}}
\newcommand{\wh}{\widehat}
\newcommand {\rb}{\rangle}
\newcommand {\lb}{{\langle}}
\newcommand {\Sw}{\mathcal{S}}
\newcommand {\F}{{\mathcal{F}}}
\newcommand {\bmo}{\mathrm{bmo}}
\newcommand {\veps}{\varepsilon}
\newcommand {\vanish}[1]{\relax}

\def\XXint#1#2#3{{\setbox0=\hbox{$#1{#2#3}{\int}$ }
\vcenter{\hbox{$#2#3$ }}\kern-.6\wd0}}

\numberwithin{equation}{section}

\makeatletter
\@namedef{subjclassname@2020}{%
  \textup{2020} Mathematics Subject Classification}
\makeatother

\title{Nonlinear wave equations with slowly decaying initial data}

\author{Jan Rozendaal}
\address{Institute of Mathematics, Polish Academy of Sciences\\
ul.~\'{S}niadeckich 8\\
00-656 Warsaw\\
Poland}
\email{jrozendaal@impan.pl}

\author[Robert Schippa]{Robert Schippa*}
\address{Karlsruhe Institute of Technology\\
Englerstrasse 2 \\
76131 Karlsruhe \\
Germany
}
\email{robert.schippa@kit.edu}



\thanks{*Corresponding author}

\begin{document}
\begin{abstract}
New local smoothing estimates in Besov spaces adapted to the half-wave group are proved via $\ell^2$-decoupling. We apply these estimates to obtain new well-posedness results for the cubic nonlinear wave equation in two dimensions. The results are compared to new well-posedness results in $L^p$-based Sobolev spaces.
\end{abstract}

\maketitle

\section{Introduction}

\subsection*{Setting}

We consider nonlinear wave equations with power-type nonlinearity:
\begin{equation}
\label{eq:NLWIntro}
\begin{cases}
\partial_t^2 u &= \Delta_{x} u \pm |u|^{\alpha-1} u,\quad (t,x) \in \R \times \R^d, \; d \geq 2,\\
u(0)&= f \in X, \quad \dot{u}(0) = g \in Y,
\end{cases}
\end{equation}
We shall analyze in detail the cubic nonlinear wave equation, where $d=2$ and $\alpha =3$. Moreover, we consider slowly decaying initial data, by which we mean initial data contained in $L^{p}$-based spaces for $p>2$.

Recently, the well-posedness of the nonlinear Schr\"odinger equation with slowly decaying initial data has attracted attention \cite{DodsonSofferSpencer2021,Schippa2022}, in part due to the importance of such initial data for modeling signals. The well-posedness results in this article are proved via a simple contraction mapping argument; similar to \cite{Schippa2022} by the second author. We use Duhamel's formula to write \eqref{eq:NLWIntro} as
\[
u(t) = \Phi_{f,g}(u) = \cos(t \sqrt{- \Delta}) f + \frac{\sin(t \sqrt{- \Delta}) g}{\sqrt{- \Delta}} \pm \int_0^t \frac{\sin((t-s) \sqrt{- \Delta})}{\sqrt{- \Delta}} (|u|^{\alpha - 1} u)(s) ds.
\]
The proof that $\Phi_{f,g}$ is a contraction in a space-time function space $S$ hinges on linear estimates
\[
\| \cos(t \sqrt{- \Delta}) f \|_S \lesssim \| f \|_X , \qquad \Big\| \frac{\sin (t \sqrt{- \Delta}) g}{\sqrt{- \Delta}} g \Big\|_S \lesssim \| g \|_Y,
\]
and a nonlinear estimate
\[
\Big\| \int_0^t \frac{\sin((t-s) \sqrt{- \Delta})}{\sqrt{ - \Delta}} (|u|^{\alpha -1} u)(s) ds \Big\|_S \lesssim \| u \|_S^\alpha.
\]
We shall use space-time Lebesgue spaces $S = L_t^r([0,T],L^p(\R^d))$ as iteration spaces; possibly intersected with another function space. As spaces of initial data, for $2<p<\infty$, we choose $X = W^{s,p}(\R^d)$ and $Y = W^{s-1,p}(\R^d)$, or we consider Besov spaces $X = \mathcal{B}^{s}_{p,2,2}(\R^d)$ and $Y = \mathcal{B}^{s-1}_{p,2,2}(\R^d)$ adapted to the half-wave group. The spaces $\Baspqr$, which are introduced in this article, are invariant under the half-wave group, and they satisfy Sobolev embeddings into the standard Besov scale. This invariance under the half-wave group is in sharp contrast with $W^{s,p}(\Rd)$ for $p\neq 2$, and a key motivation to consider adapted spaces.

\subsection*{Adapted spaces and local smoothing}

The use of adapted Besov spaces builds on recent work concerning invariant spaces for Schr\"{o}dinger and wave equations. Indeed, modulation spaces, invariant spaces for Schr\"{o}dinger propagators, have been used extensively as spaces of initial data for nonlinear Schr\"{o}dinger equations (see \cite{GuoWangZhao2006,ChaichenetsHundertmarkKunstmannPattakos2017,ChaichenetsHundertmarkKunstmannPattakos2019,Schippa2022} and references therein). On the other hand, a scale $(\Hp)_{1\leq p\leq \infty}$ of Hardy spaces for Fourier integral operators (FIOs) was introduced in \cite{HassellPortalRozendaal2020} by Hassell, Portal and the first author. This work in turn generalizes the case $p=1$ due to Smith \cite{Smith1998}, which predates \cite{HassellPortalRozendaal2020} by decades. The Hardy spaces for FIOs are invariant under half-wave propagators and more general FIOs, and they satisfy the Sobolev embeddings
\begin{equation}\label{eq:Sobolevintro}
W^{s(p),p}(\Rd)\subseteq \Hp\subseteq W^{-s(p),p}(\Rd)
\end{equation}
for all $1<p<\infty$, with the natural modifications involving the local Hardy space $\HT^{1}(\Rd)$ for $p=1$, and $\bmo(\Rd)$ for $p=\infty$. Here and throughout,
\begin{equation}\label{eq:sp}
s(p)=\frac{d-1}{2}\Big|\frac{1}{2}-\frac{1}{p}\Big|.
\end{equation}
By combining these two properties, one recovers the sharp $L^{p}$ mapping properties of the half-wave group, due to Peral and Miyachi \cite{Peral1980,Miyachi1980}:
\begin{equation}\label{eq:fixed}
e^{it\sqrt{-\Delta}}:W^{2s(p),p}(\Rd)\to L^{p}(\Rd)
\end{equation}
for all $1<p<\infty$ and $t\in\R$, and the more general $L^{p}$ mapping properties of FIOs due to Seeger, Sogge and Stein \cite{SeegerSoggeStein1991}.

The invariance of these spaces under the solution operators to Schr\"{o}dinger and wave equations allows one to use iterative constructions to build parametrices, as was done for rough wave equations using $\Hp$ in \cite{HassellRozendaal2020}. It also shows that such spaces are natural for the fixed-time regularity of these equations.

On the other hand, it was observed by Sogge \cite{Sogge1991} that considering space-time Lebesgue norms of the solution to the Euclidean wave equation yields a gain of regularity over the fixed-time estimates in \eqref{eq:fixed}. More precisely, in \cite{Sogge1991} Sogge formulated the local smoothing conjecture for the Euclidean wave equation, which states that
\begin{equation}
\label{eq:LocalSmoothingConjecture}
\| e^{it \sqrt{- \Delta}} f \|_{L_t^p([0,1],L^p(\R^d))} \lesssim_{\veps} \| f \|_{W^{\sigma(p)+\veps,p}(\R^d)}
\end{equation}
for all $\veps>0$, where $\sigma(p)=0$ for $2<p\leq 2d/(d-1)$, and $\sigma(p)=2s(p)-1/p$ for $p>2d/(d-1)$.

The local smoothing conjecture implies several open problems in harmonic analysis, like the Bochner--Riesz conjecture and the restriction conjecture. A breakthrough result was the proof of the sharp $\ell^2$-decoupling inequality for the cone, due to Bourgain--Demeter \cite{BourgainDemeter2015}. More precisely, set  
\[
\overline{s}(p):= 
\begin{cases}
0,& 2 \leq p \leq \frac{2(d+1)}{d-1},\\
s(p) - \frac{1}{p},&\frac{2(d+1)}{d-1} \leq p < \infty.
\end{cases}
\]
For $k\in\N$, let $(\chi_\nu)_{\nu}$ be a partition of unity of $\Rd\setminus \{0\}$ with smooth zero-homogeneous functions, which localize to cones of aperture aproximately $2^{-k/2}$, and let $g\in\Sw(\R)$ be such that $|g(t)|\geq 1$ for $t\in[0,1]$, and $\supp(\wh{g}\,)\subseteq[-1,1]$. In the following let $D = \sqrt{- \Delta}$ and for a measurable function $m: \R^d \to \R$, we denote the Fourier multiplier by $m(D)$ such that $(m(D) f) \widehat (\xi) = m(\xi) \hat{f}(\xi)$. After rescaling to unit frequencies (see e.g.~\cite[Section~3]{BeltranHickmanSogge2020}), using that $\|e^{it\sqrt{-\Delta}}f\|_{L^{p}([0,1]\times\Rd)}\leq \|g(t)e^{it\sqrt{-\Delta}}f\|_{L^{p}(\R\times\Rd)}$ and that $(t,x)\mapsto g(t)e^{it\sqrt{-\Delta}}f(x)$ has frequency support near the light cone, it then follows from the $\ell^{2}$-decoupling inequality \cite[Theorem~1.2]{BourgainDemeter2015} that
\begin{equation}
\label{eq:DecouplingConeIntroduction}
\|e^{it \sqrt{- \Delta}}f \|_{L^p([0,1]\times \Rd)} \lesssim 2^{k(\overline{s}(p) + \varepsilon)} \Big(\sum_\nu \|g(t)e^{it \sqrt{- \Delta}} \chi_\nu(D) f \|^2_{L^p(\R\times\Rd)} \Big)^{1/2}
\end{equation}
for any $\veps>0$ and $f\in L^{p}(\Rd)$ with $\supp(\wh{f}\,)\subseteq \{\xi\in\Rd\mid 2^{-1+k}\leq |\xi|\leq 2^{1+k}\}$.  
In turn, from \eqref{eq:DecouplingConeIntroduction} follow local smoothing estimates by an application of H\"older's inequality, to pass from the $\ell^2$ norm  to the $\ell^p$ norm, and from a kernel estimate. 
Although \eqref{eq:DecouplingConeIntroduction} is sharp, it does not imply the local smoothing conjecture; it only yields the required bounds for $p\geq 2(d+1)/(d-1)$. The local smoothing conjecture was recently resolved for $d=2$ 
via a sharp (reverse) $L^4$-square function estimate by Guth--Wang--Zhang \cite{GuthWangZhang2020}, but it is still open for $d\geq 3$. 

Coming back to the nonlinear wave equation, we will use local smoothing estimates to lower the regularity of the initial data required to solve  \eqref{eq:NLWIntro}, thereby providing, to the best of the authors' knowledge, a novel approach to nonlinear wave equations with slowly decaying initial data.

\subsection*{Main results}

Firstly, we introduce the adapted Besov spaces $\Baspqr$ and we derive some of their properties. In particular, we show the Besov counterpart of the Sobolev embeddings in \eqref{eq:Sobolevintro}:
\begin{equation}\label{eq:SobolevBesovIntro}
B^{s+s(p)}_{p,p}(\Rd)\subseteq\Basppp\subseteq B^{s-s(p)}_{p,p}(\Rd).
\end{equation}
We also show the invariance of $\Baspqq$ under the half-wave propagators. In fact, we take this opportunity to show the sharp polynomial growth rate of the $\Baspqq$-norm under evolution of the half-wave group. This quantifies a polynomial growth result by the first author \cite[Lemma~3.5]{Rozendaal2021LocalSmoothing}, which was established for the Hardy spaces for FIOs. More precisely, in Proposition \ref{prop:PolynomialGrowth} we show that 
\begin{equation}
\label{eq:PolynomialGrowthIntro}
\| e^{it \sqrt{- \Delta}} f \|_{\mathcal{B}^{s}_{p,q,q}(\R^d)} \lesssim (1+|t|)^{2s(p)} \| f \|_{\mathcal{B}^{s}_{p,q,q}(\R^d)}
\end{equation}
for all $p\in[1,\infty]$, $q\in[1,\infty)$, $s,t\in\R$ and $f\in\Baspqq$. In Proposition \ref{prop:SharpGrowth} we show that this is sharp, by using a radial Knapp example. 

Next, we obtain improved local smoothing estimates, in terms of $\Ba^{s}_{p,2,2}(\Rd)$. 

\begin{theorem}
\label{thm:LocalSmoothingHardy}
Let $d \geq 2$, $p\in(2,\infty)$ and $\veps>0$. Then there exists a $C\geq0$ such that
\begin{equation}
\label{eq:LocalSmoothingHardySpaces}
\| e^{it \sqrt{- \Delta}} f \|_{L_t^p([0,1],L^p(\R^d))} \leq C \| f \|_{\mathcal{B}^{\overline{s}(p)+\veps}_{p,2,2}(\R^d)}
\end{equation}
for all $f\in\Ba^{s}_{p,2,2}(\Rd)$.
\end{theorem}

The exponent $\overline{s}(p)$ in \eqref{eq:LocalSmoothingHardySpaces} is sharp for all $2<p<\infty$, cf.~Remark \ref{rem:sharpness2}. In fact, the right-hand side of \eqref{eq:DecouplingConeIntroduction} is equivalent to the $\Ba^{\overline{s}(p)+\veps}_{p,2,2}(\Rd)$-norm. Hence, when restricted to dyadic frequency annuli, $\Ba^{\overline{s}(p)+\veps}_{p,2,2}(\Rd)$ is the largest space of initial data for which one can obtain local smoothing estimates when applying the $\ell^{2}$-decoupling inequality in the manner in which it is typically used. 

The corresponding bounds for $\Hps$, or equivalently for $\Basppp$, are due to the first author \cite{Rozendaal2021LocalSmoothing}. We note that
\begin{equation}
\label{eq:EmbeddingRefinement}
W^{s+s(p)+2\veps,p}(\Rd)\subseteq \mathcal{B}^{s+\varepsilon}_{p,p,p}(\Rd) \subseteq \mathcal{B}^{s}_{p,2,2}(\Rd)
\end{equation}
for all $s\in\R$, $2<p<\infty$ and $\veps>0$, and that the $\Ba^{s}_{p,p,p}(\Rd)$-norm of certain functions is substantially larger than their $\Ba^{s}_{p,2,2}(\Rd)$-norm (see Remark \ref{rem:dyadicpar}). Hence \eqref{eq:LocalSmoothingHardySpaces} strictly improves upon the bounds in \cite{Rozendaal2021LocalSmoothing}, and in particular upon the local smoothing conjecture for $p\geq 2(d+1)/(d-1)$. On the other hand, it is an open question whether $\Ba^{s}_{p,2,2}(\Rd)$ is invariant under general FIOs, as $\Hps$ is. For $2<p<2(d+1)/(d-1)$, \eqref{eq:LocalSmoothingHardySpaces} neither follows from the local smoothing conjecture, nor does it imply it.

Next, we show how local smoothing estimates can be combined with nonlinear Strichartz estimates to prove well-posedness for nonlinear wave equations with slowly decaying initial data. We write $\dot{H}^s(\R^d) = |D|^{-s} L^2(\R^d)$.

\begin{theorem}
\label{thm:LWPNLW}
Let $d=2$, $\alpha = 3$, and $\varepsilon > 0$. Then, \eqref{eq:NLWIntro} is analytically locally well posed with initial data space 
\begin{equation}\label{eq:LWPNLW1}
X \times Y = (\mathcal{B}^{\varepsilon}_{4,2,2}(\R^2) + \dot{H}^{3/8}(\R^2)) \times (\mathcal{B}^{\varepsilon-1}_{4,2,2}(\R^2) + \dot{H}^{-5/8}(\R^2)),
\end{equation}
and solution space $S_{T}= L_t^{24/7}([0,T],L^4(\R^2)) \cap C([0,T],\mathcal{B}^{\varepsilon}_{4,2,2}(\R^2) + \dot{H}^{3/8}(\R^2) )$ with $T = T(\|(f,g) \|_{X \times Y})$. Moreover, \eqref{eq:NLWIntro} is analytically locally well posed with initial data space
\begin{equation}\label{eq:LWPNLW2}
X \times Y = (\mathcal{B}^{\varepsilon}_{6,2,2}(\R^2) + \dot{H}^{1/2}(\R^2)) \times (\mathcal{B}^{\varepsilon-1}_{6,2,2}(\R^2) + \dot{H}^{-1/2}(\R^2)),
\end{equation}
and solution space $S_{T} = L_t^4([0,T],L^6(\R^2)) \cap C([0,T],\mathcal{B}^{\varepsilon}_{6,2,2}(\R^2) + \dot{H}^{1/2}(\R^2) )$.
%
\end{theorem}

For the definition of analytic well-posedness we refer to Section \ref{subsec:prelim}. Roughly speaking, for each $(f,g)\in X\times Y$ we obtain a time of existence  $T(\|(f,g)\|_{X\times Y})$, for which there is a unique solution $u$ in $S_{T}$ which depends analytically on the initial data. Moreover, for any $T>0$, there exists some $\veps=\veps(T)>0$ such that \eqref{eq:NLWIntro} is well posed in $S_{T}$ whenever $\|(f,g)\|_{X\times Y}<\veps$.

\begin{remark}
\label{rem:LWP}
The homogeneous Sobolev spaces $\dot{H}^{3/8}(\R^2)$ and $\dot{H}^{-5/8}(\R^2)$ in \eqref{eq:LWPNLW1} can be replaced by the inhomogeneous spaces $H^s(\R^2)=W^{s,2}(\R^2)$ and $H^{s-1}(\R^2)$, for $s \geq 3/8$, and similarly for \eqref{eq:LWPNLW2}.
Furthermore, the arguments from the proof yield local well-posedness for initial data in
\begin{equation}\label{eq:L4data}
X \times Y = W^{\veps,4}(\R^2) \times W^{-1+\veps,4}(\R^2)
\end{equation}
with solutions in $S_T = L_t^{24/7}([0,T],L^4(\R^2))$, and for
\begin{equation*}
X \times Y = W^{1/6+\veps,6}(\R^2) \times W^{-5/6+\veps,6}(\R^2), 
\end{equation*}
with solution space $S_T = L_t^4([0,T],L^6(\R^2))$. By the embeddings in \eqref{eq:EmbeddingRefinement}, we have
\begin{equation*}
W^{1/6+3\veps}(\R^2) \subseteq \mathcal{B}^{2\veps}_{6,6,6}(\R^2) \subseteq \mathcal{B}^{\veps}_{6,2,2}(\R^2),
\end{equation*}
which shows that a local well-posedness result with initial data in $\Ba^{s+\veps}_{6,2,2}(\R^2)$ supersedes one involving $W^{s+1/6+\veps}(\R^2)$. 
This is not quite the case for the $L^4$-based result because one has the sharp embeddings
\begin{equation*}
W^{1/8+3\varepsilon,4}(\R^2) \subseteq \mathcal{B}^{2\varepsilon}_{4,4,4}(\R^2) \subseteq \mathcal{B}^\varepsilon_{4,2,2}(\R^2).
\end{equation*}
It appears that this mismatch of $1/8$ derivatives reflects the fact that $\ell^2$-decoupling does not imply the local smoothing conjecture. It would be very interesting to eventually translate this additional smoothing effect to adapted function spaces. However, it follows from Proposition \ref{prop:Sobolev2} and from the sharpness of the results in \cite{Rozendaal2021LocalSmoothing} that such an additional smoothing effect cannot be captured by $\Ba^{s}_{p,q,q}(\Rd)$ for $q\leq p$ (see also Remark \ref{rem:sharpness2}). We do note that, since \eqref{eq:LocalSmoothingHardySpaces} complements the local smoothing conjecture, the well-posedness result in Theorem \ref{thm:LWPNLW} neither follows from one involving \eqref{eq:L4data}, nor does it imply it.
\end{remark}
We also show local well-posedness for slower decaying initial data, i.e., with initial data in spaces $\mathcal{B}^{s(p)}_{p,2,2}(\R^d) \times \mathcal{B}^{s(p)-1}_{p,2,2}(\R^d)$ and $p=4n+2$, $n \in \Z_{\geq 2}$, $d \in \{2,3\}$. The more technical result is stated in Theorem \ref{thm:SlowlyDecayingData}. 

In Theorems \ref{thm:GlobalNLW} and \ref{thm:GWPSlowerDecay} we prove global well-posedness in the defocusing case. Global well-posedness in $L^2$-based Sobolev spaces typically follows from conserved quantity. This does not fit well into the $L^p$-scale. Instead, we show global well-posedness by adapting arguments of Dodson--Soffer--Spencer \cite{DodsonSofferSpencer2021}; see also \cite{Schippa2022,KlausKunstmann2021}.

\subsection*{Generalizations}

The presented arguments are robust in nature and allow one to treat more general nonlinearities than $\alpha=3$ (see, e.g.,~Theorem \ref{thm:LWPQuinticNLW}). One can also consider higher dimensions $d \geq 3$, albeit in this case with a different derivative parameter $s$.

Moreover, the arguments transpire to the variable-coefficient case. Consider the nonlinear wave equation on a compact Riemannian manifold $(M,g)$ with $\dim M \geq 2$:
\begin{equation}
\label{eq:VariableCoefficientWaveEquation}
\begin{cases}
\partial_t^2 u &= \Delta_g u \pm |u|^{\alpha - 1} u, \quad (t,x) \in \R \times M, \\
u(0) &= f_1 \in X, \quad \dot{u}(0) = f_2 \in Y.
\end{cases}
\end{equation}
For $X \times Y \in W^{s,p}(M) \times W^{s-1,p}(M)$, $\dim M =2$, we can argue as in the proof of Theorem \ref{thm:LWPNLW}, because both local smoothing and Strichartz estimates remain true in the variable coefficient case. Indeed, variable-coefficient decoupling was proved by Beltran--Hickman--Sogge \cite{BeltranHickmanSogge2020} and local-in-time Strichartz estimates remain true on compact manifolds, as proved by Kapitanskii \cite{Kapitanskii1989,Kapitanskii1989II}. These are the key ingredients for the iteration argument in Section \ref{section:LWPNLW}. 

For an extension of the results on nonlinear equations with initial data in adapted spaces, one would have to find a suitable definition on compact manifolds. On the other hand, to prove global results it seems natural to work with spaces of initial data which are invariant under more general FIOs. Indeed, the solution operator to the linear part of \eqref{eq:VariableCoefficientWaveEquation} is a Fourier integral operator, an observation which goes back to Lax \cite{Lax1957}. This motivated the pioneering works by Seeger--Sogge--Stein \cite{SeegerSoggeStein1991} and Mockenhaupt--Seeger--Sogge \cite{MockenhauptSeegerSogge1993} on the fixed-time and space-time mapping properties of FIOs.  It is unclear whether $\mathcal{B}^s_{p,2,2}(\Rd)$ is invariant under more general FIOs, but $\Hps$ is. Moreover, one could solve nonlinear wave equations with initial data in $\Hps$ in the same manner as we do for $\mathcal{B}^s_{p,2,2}(\Rd)$.

Our goal in this article is not to develop a full theory of Besov spaces adapted to the half-wave group, as has been done for the Hardy spaces for FIOs in \cite{HassellPortalRozendaal2020,Rozendaal2021,FaLiRoSo19}. The advantage of working with Besov spaces is that it suffices to obtain estimates on dyadic frequency annuli, instead of working with square functions. On the other hand, one only recovers the sharp fixed-time regularity for wave equations in the Besov scale, cf.~\eqref{eq:SobolevBesovIntro}, as opposed to the $L^{p}$-scale, cf.~\eqref{eq:Sobolevintro}.

\subsection*{Organization} In Section \ref{section:FunctionSpaces} we introduce the function spaces $\Baspqr$, and we determine some of their properties. In Section \ref{sec:invariance} we show that $\Baspqq$ is invariant under the action of the half-wave group, cf.~\eqref{eq:PolynomialGrowthIntro}, and we obtain product estimates. In Section \ref{section:LocalSmoothing} we prove the local smoothing estimates in Theorem \ref{thm:LocalSmoothingHardy}. Using these, in Section \ref{section:LWPNLW} we derive local well-posedness results, and in particular Theorem \ref{thm:LWPNLW}. In Section \ref{subsection:GlobalNLW} we use a blow-up alternative to prove global well-posedness in the defocusing case.

\subsection*{Notation}

The natural numbers are $\N:=\{1,2,\ldots\}$, and we write $\N_{0}:=\N\cup\{0\}$. Throughout most of this article we fix a general dimension $d\in\N$ with $d\geq2$, but in Section \ref{section:LWPNLW} we will typically assume that $d\in \{2,3\}$.

For $\xi\in\Rd$ we write $\lb\xi\rb=(1+|\xi|^{2})^{1/2}$, and $\hat{\xi}=\xi/|\xi|$ if $\xi\neq0$. We use multi-index notation, where $\partial_{\xi}=(\partial_{\xi_{1}},\ldots,\partial_{\xi_{d}})$ and $\partial^{\alpha}_{\xi}=\partial^{\alpha_{1}}_{\xi_{1}}\ldots\partial^{\alpha_{d}}_{\xi_{d}}$  
for $\xi=(\xi_{1},\ldots,\xi_{d})\in\Rd$ and $\alpha=(\alpha_{1},\ldots,\alpha_{d})\in\N_{0}^{d}$. The Fourier transform of $f\in\Sw'(\Rd)$ is denoted by $\F f$ or $\widehat{f}$, and the Fourier multiplier with symbol $\ph\in\Sw'(\Rd)$ is denoted by $\ph(D)$. 

We write $f(s)\lesssim g(s)$ to indicate that $f(s)\leq Cg(s)$ for all $s$ and a constant $C>0$ independent of $s$, and similarly for $f(s)\gtrsim g(s)$ and $g(s)\eqsim f(s)$.

\section{Function spaces}
\label{section:FunctionSpaces}

In this section we introduce the relevant function spaces for this article, and we derive some of their properties, most notably equivalent norms and embeddings. 

\subsection{Definitions}\label{subsec:def}

We first recall the definition of the Hardy spaces for FIOs from  \cite{Smith1998,HassellPortalRozendaal2020}. Fix a non-negative radial $\varphi \in C^\infty_c(\R^d)$ such that $\varphi(\xi) = 0$ for $|\xi| > 1$, and $\varphi \equiv 1$ in a neighbourhood of zero. For $\omega \in \mathbb{S}^{n-1}$, $\sigma > 0$, and $\xi \in \R^d \setminus \{0\}$, set $\varphi_{\omega,\sigma}(\xi) := c_\sigma \varphi \big( \frac{\hat{\xi} - \omega}{\sigma^{1/2}} \big)$, where $c_\sigma := \big( \int_{\mathbb{S}^{d-1}} \varphi \big( \frac{e_1 - \nu}{\sigma^{1/2}} \big)^2 d\nu \big)^{-1/2}$ for $e_1 = (1,0,\ldots,0)$. Furthermore, we set $\varphi_{\omega,\sigma}(0) := 0$. Let $\psi \in C^\infty_c(\R^d)$ be a non-negative radial function such that $\psi(\xi) = 0$ if $|\xi| \notin [1/2,2]$, with $\psi(\xi) = 0$ if $|\xi| \notin [1/2,2]$, and
\begin{equation*}
\int_0^\infty \psi(\sigma \xi)^2 \frac{d \sigma}{\sigma} = 1 \text{ for all } \xi \neq 0.
\end{equation*}
Let $\varphi_\omega(\xi) := \int_0^4 \psi(\sigma \xi) \varphi_{\omega, \sigma}(\xi) \frac{d \sigma}{\sigma}$. Recall the following properties of $\ph_{\w}\in C^{\infty}(\Rd)$ from \cite[Remark 3.3]{Rozendaal2021}:
\begin{itemize}
\item[(1)] For all $\omega \in \mathbb{S}^{d-1}$ and $\xi \neq 0$ one has $\varphi_\omega(\xi) = 0$ if $|\xi| < 1/8$ or $|\hat{\xi} - \omega| > 2 |\xi|^{-1/2}$.
\item[(2)] For all $\alpha \in \N^d_0$ and $\beta \in \N_0$ there exists $C_{\alpha,\beta} \geq 0$ such that
\begin{equation*}
| (\omega \cdot \partial_\xi)^\beta \partial_\xi^\alpha \varphi_\omega(\xi)| \leq C_{\alpha,\beta} |\xi|^{\frac{d-1}{4}-\frac{|\alpha|}{2}-\beta}
\end{equation*}
for all $\omega \in \mathbb{S}^{d-1}$ and $\xi \neq 0$.
\item[(3)] For all $\alpha \in \N^d_0$ there exists a $C_\alpha \geq 0$ such that
\begin{equation*}
\Big| \partial_\xi^\alpha \Big( \int_{\mathbb{S}^{d-1}} \varphi_\omega(\xi) d\omega \Big)^{-1} \Big| \leq C_\alpha |\xi|^{\frac{d-1}{4}-|\alpha|}
\end{equation*}
for all $\xi \in \R^d$ with $|\xi| \geq 1/2$. Hence there is an $m \in S^{(d-1)/4}(\R^d)$ such that, if $f \in \mathcal{S}'(\R^d)$ satisfies $\text{supp}(\hat{f}) \subseteq \{ \xi \in \Rd\mid |\xi| \geq 1/2 \}$, then
\begin{equation}\label{eq:repro}
f = \int m(D) \varphi_\nu(D) f d\nu.
\end{equation}
\end{itemize}
For simplicity of notation, we write $\mathcal{H}^p(\R^d) = L^p(\R^d)$ for $1<p<\infty$, and 
$\mathcal{H}^1(\R^d)$ is the classical local Hardy space.  Fix a $\rho\in C^{\infty}_{c}(\Rd)$ such that $\rho(\xi)=1$ for $|\xi|\leq 2$.

We define the Hardy spaces for FIOs as follows.

\begin{definition}\label{def:HpFIO}
For $p \in [1,\infty)$ and $s\in\R$, let $\Hps$ consist of all $f \in \mathcal{S}'(\R^d)$ such that $\rho(D) f \in L^p(\Rd)$, $\langle D \rangle^s \varphi_\omega(D) f \in \mathcal{H}^{p}(\R^d)$ for almost all $\omega \in \mathbb{S}^{d-1}$, and
\begin{equation*}
\| f \|_{\Hps} := \| \rho(D) f \|_{L^p(\R^d)} + \Big( \int_{\mathbb{S}^{d-1}} \| \langle D \rangle^s \varphi_\omega(D) f \|^p_{\mathcal{H}^p(\R^d)} d\omega \Big)^{1/p} < \infty.
\end{equation*}
Moreover, $\HT^{s,\infty}_{FIO}(\Rd):=(\HT^{-s,1}_{FIO}(\Rd))^{*}$.
\end{definition}

In fact, $\Hp$ was originally defined in \cite{Smith1998,HassellPortalRozendaal2020} using conical square function estimates over the cosphere bundle. This includes an intrinsic definition of $\HT^{\infty}_{FIO}(\Rd)$ in terms of Carleson measures. The equivalent characterization in Definition \ref{def:HpFIO} was obtained in \cite{FaLiRoSo19,Rozendaal2021}.

Recall that, for $p,r\in[1,\infty]$ and $s\in\R$, the Besov space $\Bspq$ consists of those $f\in\Sw'(\Rd)$ such that
\[
\|f\|_{\Bspq}:=\Big(\sum_{k=0}^{\infty}2^{srk}\|\psi_{k}(D)f\|_{L^{p}(\Rd)}^{r}\Big)^{1/r}<\infty.
\]
Here and throughout,  $(\psi_{k})_{k=0}^{\infty}\subseteq C^{\infty}_{c}(\Rd)$ is a fixed Littlewood-Paley decomposition, with $\supp(\psi_{k})\subseteq \{\xi\in\Rd\mid 2^{k-1}\leq |\xi|\leq 2^{k+1}\}$ for $k\in\N$, and $\sum_{k=0}^{\infty}\psi_{k}(\xi)=1$ for all $\xi\neq 0$. 

We consider the following Besov variant of the Hardy spaces for FIOs.

\begin{definition}\label{def:Besov}
Let $p,r\in[1,\infty]$, $q\in [1,\infty)$ and $s\in\R$, Then $\Baspqr$ consists of all $f \in \mathcal{S}'(\R^d)$ such that $\rho(D) f \in L^p(\Rd)$, $ \varphi_\omega(D) f \in \Bspq$ for almost all $\omega \in \mathbb{S}^{d-1}$, and
\[
\| f \|_{\Baspqr} := \| \rho(D) f \|_{L^p(\R^d)} + \Big( \int_{\mathbb{S}^{d-1}} \|\varphi_\omega(D) f \|^q_{B^{s}_{p,r}(\Rd)} d\omega \Big)^{1/q} < \infty.
\]
\end{definition}

We will mostly deal with the case where $q=r$. Then 
the following lemma allows one to reduce various arguments to dyadic frequency annuli. 

\begin{lemma}\label{lem:dyadic}
Let $p\in[1,\infty]$, $q\in[1,\infty)$ and $s\in\R$. Then there exists a $C>0$ such that the following holds. An $f\in\Sw'(\Rd)$ satisfies $f\in \Baspqq$ if and only if $(\sum_{k=0}^{\infty}\|\psi_{k}(D)f\|_{\Baspqq}^{q})^{1/q}<\infty$, in which case
\begin{equation}\label{eq:dyadic}
\frac{1}{C}\|f\|_{\Baspqq}\leq \Big(\sum_{k=0}^{\infty}\|\psi_{k}(D)f\|_{\Baspqq}^{q}\Big)^{1/q}\leq C\|f\|_{\Baspqq}.
\end{equation}
\end{lemma}
\begin{proof}
First note that
\begin{equation}\label{eq:littleHolder}
\|g\|_{\Baspqq}\eqsim \Big(\| \rho(D) f \|_{L^p(\R^d)}^{q} + \int_{\mathbb{S}^{d-1}} \|\varphi_\omega(D) f \|^q_{B^{s}_{p,q}(\Rd)} d\omega \Big)^{1/q}
\end{equation}
for all $g\in\Sw'(\Rd)$ such that either of these quantities is finite.

Let $f\in\Baspqq$. Using \eqref{eq:littleHolder} and Fubini's lemma, one can bound the middle term in \eqref{eq:dyadic} by a multiple of
\[
\Big(\sum_{k=0}^{\infty}\|\rho(D)\psi_{k}(D)f\|_{L^{p}(\Rd)}^{q}+\int_{\mathbb{S}^{d-1}}\sum_{l=0}^{\infty}2^{sql}\|\psi_{l}(D)\ph_{\w}(D)\psi_{k}(D)f\|_{L^{p}(\Rd)}^{q}d\w\Big)^{1/q}.
\]
Moreover, there exists a $k_{0}\in\Z_{+}$ such that $\psi_{k}(D)\rho(D)f=0$ for all $k>k_{0}$. Since the $\F^{-1}(\psi_{k})$ are uniformly bounded in $L^{1}(\Rd)$, we thus obtain
\[
\sum_{k=0}^{\infty}\|\rho(D)\psi_{k}(D)f\|_{L^{p}(\Rd)}^{q}\lesssim \sum_{k=0}^{k_{0}}\|\rho(D)f\|_{L^{p}(\Rd)}^{q}\lesssim \|\rho(D)f\|_{L^{p}(\Rd)}^{q}.
\] 
Using also the support properties of the $\psi_{k}$, one obtains
\begin{align*}
&\quad\sum_{k=0}^{\infty}\int_{\mathbb{S}^{d-1}}\sum_{l=0}^{\infty}2^{sql}\|\psi_{l}(D)\ph_{\w}(D)\psi_{k}(D)f\|_{L^{p}(\Rd)}^{q}d\w\\
&=\int_{\mathbb{S}^{d-1}}\sum_{l=0}^{\infty}2^{sql}\sum_{k=l-1}^{l+1}\|\psi_{k}(D)\psi_{l}(D)\ph_{\w}(D)f\|_{L^{p}(\Rd)}^{q}d\w\\
&\lesssim \int_{\mathbb{S}^{d-1}}\sum_{l=0}^{\infty}2^{sql}\|\psi_{l}(D)\ph_{\w}(D)f\|_{L^{p}(\Rd)}^{q}d\w=\int_{\mathbb{S}^{d-1}}\|\ph_{\w}(D)f\|_{B^{s}_{p,q}(\Rd)}^{q}d\w.
\end{align*}
Combined with \eqref{eq:littleHolder}, this proves the second inequality in \eqref{eq:dyadic}.

On the other hand, for each $f\in\Sw'(\Rd)$ one has $f=\sum_{k=0}^{\infty}\psi_{k}(D)f$, and thus
\[
\|\rho(D)f\|_{L^{p}(\Rd)}\leq \sum_{k=0}^{k_{0}}\|\rho(D)\psi_{k}(D)f\|_{L^{p}(\Rd)}\lesssim \Big(\sum_{k=0}^{\infty}\|\rho(D)\psi_{k}(D)f\|_{L^{p}(\Rd)}^{q}\Big)^{1/q}.
\]
Similarly, H\"{o}lder's inequality and the support properties of the $\psi_{k}$ combine to yield
\begin{align*}
&\quad\int_{\mathbb{S}^{d-1}}\|\ph_{\w}(D)f\|_{B^{s}_{p,q}(\Rd)}^{q}d\w=\int_{\mathbb{S}^{d-1}}\sum_{l=0}^{\infty}2^{sql}\Big\|\sum_{k=l-1}^{k+1}\psi_{l}(D)\ph_{\w}(D)\psi_{k}(D)f\Big\|_{L^{p}(\Rd)}^{q}d\w\\
&\lesssim \sum_{k=0}^{\infty}\int_{\mathbb{S}^{d-1}}\sum_{l=0}^{\infty}2^{sql}\|\psi_{l}(D)\ph_{\w}(D)\psi_{k}(D)f\|_{L^{p}(\Rd)}^{q}d\w\\
&=\sum_{k=0}^{\infty}\int_{\mathbb{S}^{d-1}}\|\ph_{\w}(D)\psi_{k}(D)f\|_{B^{s}_{p,q}(\Rd)}^{q}d\w.
\end{align*}
Together with \eqref{eq:littleHolder}, this proves the first inequality in \eqref{eq:dyadic}.
\end{proof}

\subsection{An equivalent norm}\label{subsec:norm}

For each $k\in \N_{0}$, fix a maximal collection $\Theta_{k}\subseteq \Su^{d-1}$ of unit vectors such that $|\nu-\nu'|\geq 2^{-k/2}$ for all $\nu,\nu'\in \Theta_{k}$. Let $(\chi_{\nu})_{\nu\in\Theta_{k}}\subseteq C^{\infty}(\Rd\setminus\{0\})$ be an associated partition of unity. That is, each $\chi_{\nu}$ is homogeneous of order $0$ and satisfies $0\leq \chi_{\nu}\leq 1$ and $\supp(\chi_{\nu})\subseteq\{\xi\in\Rd\mid |\hat{\xi}-\nu|\leq 2^{1-k/2}\}$. Moreover, $\sum_{\nu\in \Theta_{k}}\chi_{\nu}(\xi)=1$ for all $\xi\neq0$, and for all $\alpha\in\N_{0}^{d}$ and $\beta\in\N_{0}$ there exists a $C_{\alpha,\beta}\geq0$ independent of $N$ such that, if $2^{-1+k}\leq |\xi|\leq 2^{1+k}$, then
\[
|(\hat{\xi}\cdot\partial_{\xi})^{\beta}\partial_{\xi}^{\alpha}\chi_{\nu}(\xi)|\leq C_{\alpha,\beta}2^{-k(|\alpha|/2+\beta)}
\]
for all $\nu\in \Theta_{k}$. Also write $\chi_{\nu}^{k}:=\chi_{\nu}\psi_{k}$ for $k\in\N_{0}$ and $\nu\in\Theta_{k}$, so that 
\begin{equation}\label{eq:decompose}
f=\sum_{k=0}^{\infty}\sum_{\nu\in\Theta_{k}}\chi_{\nu}^{k}(D)f
\end{equation}
for $f\in\Sw'(\Rd)$. It follows from integration by parts that
\begin{equation}\label{eq:kernelbounds}
\|\F^{-1}(\chi_{\nu}^{k})\|_{L^{p}(\Rd)}\lesssim 2^{k\frac{n+1}{2p'}}
\end{equation}
for all $p\in[1,\infty]$, with an implicit constant independent of $k$ and $\nu$.

We can now give a discrete description of the $\Baspqq$-norm. For $p=q$, the first statement in the following proposition is \cite[Proposition 4.1]{Rozendaal2021LocalSmoothing}.

\begin{proposition}\label{prop:discrete}
Let $p\in[1,\infty]$, $q\in[1,\infty)$ and $s \in \R$. Then there exists a $C>0$ such that the following holds. Let $f\in \Sw'(\Rd)$ be such that $\supp(\wh{f}\,)\subseteq\{\xi\in\Rd\mid 2^{k-1}\leq |\xi|\leq 2^{k+1}\}$ for some $k\in\N$. Then
\[
\frac{1}{C}\|f\|_{\Baspqq}\leq 2^{k(s+\frac{d-1}{2}(\frac{1}{2}-\frac{1}{q}))}\Big(\sum_{\nu\in\Theta_{k}}\|\chi_{\nu}(D)f\|_{L^{p}(\Rd)}^{q}\Big)^{1/q}\leq C\|f\|_{\Baspqq},
\]
whenever one of these quantities is finite. Hence an $f \in \Sw'(\Rd)$ satisfies $f\in\Baspqq$ if and only if 
\begin{equation}\label{eq:equivalent}
\Big(\sum_{k=0}^{\infty}2^{qk(s+\frac{d-1}{2}( \frac{1}{2} - \frac{1}{q}))}\sum_{\nu\in\Theta_{k}}\| \chi_{\nu}^{k}(D) f \|^q_{L^{p}(\Rd)}\Big)^{1/q}
\end{equation}
is finite, and \eqref{eq:equivalent} defines an equivalent norm on $\Baspqq$.
\end{proposition}
\begin{proof}
It is straightforward to deal with the low frequencies, so we may assume that $\rho(D)f=0$. Moreover, by Lemma \ref{lem:dyadic}, the second statement follows from the first.

To prove the first statement, for each $\nu\in\Theta_{k}$, set
\begin{equation}\label{eq:packets}
\tilde{\chi}_{\nu}:=\sum_{|\nu'-\nu|\leq 2^{2-k/2}}\chi_{\nu'}\quad\text{and}\quad\tilde{\chi}_\nu^{k} := \sum_{l=k-1}^{k+1} \psi_{l}\tilde{\chi}_\nu,
\end{equation}
as well as $E_{\nu}:=\{\w\in \Su^{d-1}\mid |\w-\nu|\leq 2^{3-k/2}\}$. Then $\chi_{\nu}(D)f=\tilde{\chi}_{\nu}(D)\chi_{\nu}(D)f$ and 
\begin{align*}
\| \varphi_\omega(D) \chi_\nu(D) f \|_{B^{s}_{p,q}(\Rd)}&=\| \varphi_\omega(D) \tilde{\chi}_{\nu}(D)\chi_\nu(D) f \|_{B^{s}_{p,q}(\Rd)} \\
&\lesssim 2^{k(s+\frac{d-1}{4})} \| \chi_\nu(D) f \|_{L^{p}(\Rd)}
\end{align*}
for each $\w\in E_{\nu}$, by a kernel estimate. More precisely, one uses that, if $l\in\{k-1,k,k+1\}$, then $2^{-k(d-1)/4}\varphi_\omega \psi_{l}\tilde{\chi}_\nu$ satisfies the same bounds as $\chi^k_\nu$, in \eqref{eq:kernelbounds}. Hence
\begin{equation*}
\begin{split}
\Big( \int_{\mathbb{S}^{d-1}} \|\ph_\w(D) f \|^q_{B^{s}_{p,q}(\Rd)} d\omega \Big)^{1/q}&\lesssim  \Big( \sum_{\nu\in\Theta_{k}} \int_{E_{\nu}} \| \varphi_\omega(D)\chi_{\nu}(D)f \|_{B^{s}_{p,q}(\Rd)}^q d\omega \Big)^{1/q} \\
&\lesssim 2^{k(s+\frac{d-1}{4})} \Big( \sum_{\nu \in\Theta_{k}} |E_\nu|\, \| \chi_\nu(D) f \|^q_{L^p(\Rd)} \Big)^{1/q} \\
&\lesssim 2^{k(s+\frac{d-1}{2}( \frac{1}{2} - \frac{1}{q}))} \Big( \sum_{\nu \in\Theta_{k}} \| \chi_\nu(D) f \|^q_{L^p(\Rd)} \Big)^{1/q}.
\end{split}
\end{equation*}
This proves the first inequality in the first statement. 

For the other inequality, recall that the $\tilde{\chi}_{\nu}^{k}(D)\chi_{\nu}(D)$, $\nu\in\Theta_{k}$, satisfy bounds as in \eqref{eq:kernelbounds}. Hence we can use \eqref{eq:repro}, the triangle inequality, H\"{o}lder's inequality, and the condition on the frequency support of $f$, to write
\begin{align*}
&\quad\Big(\sum_{\nu\in\Theta_{k}}\|\chi_{\nu}(D)f\|_{L^{p}(\Rd)}^{q}\Big)^{1/q}=\Big(\sum_{\nu\in\Theta_{k}}\Big\|\int_{E_{\nu}}m(D)\ph_{\w}(D)\chi_{\nu}(D)fd\w\Big\|_{L^{p}(\Rd)}^{q}\Big)^{1/q}\\
&\lesssim \Big(\sum_{\nu\in\Theta_{k}}|E_{\nu}|^{q/q'}\int_{E_{\nu}}\|m(D)\ph_{\w}(D)\tilde{\chi}^{k}_{\nu}(D)\chi_{\nu}(D)f\|_{L^{p}(\Rd)}^{q}d\w\Big)^{1/q}\\
&\lesssim 2^{-k\frac{d-1}{2q'}}\Big(\sum_{\nu\in\Theta_{k}}2^{qk\frac{d-1}{4}}\int_{E_{\nu}}\|\ph_{\w}(D)f\|_{L^{p}(\Rd)}^{q}d\w\Big)^{1/q}\\
&\eqsim 2^{k(s+\frac{d-1}{2}(\frac{1}{q}-\frac{1}{2}))}\Big(\sum_{\nu\in\Theta_{k}}\int_{E_{\nu}}\|\ph_{\w}(D)f\|_{B^{s}_{p,q}(\Rd)}^{q}d\w\Big)^{1/q}.
\end{align*}
Finally, we can use that the $\nu\in \Theta_{k}$ are $2^{-k/2}$-separated to write
\[
\Big(\sum_{\nu\in\Theta_{k}}\int_{E_{\nu}}\|\ph_{\w}(D)f\|_{B^{s}_{p,q}(\Rd)}^{q}d\w\Big)^{1/q}\lesssim\Big(\int_{\Sd}\|\ph_{\w}(D)f\|_{B^{s}_{p,q}(\Rd)}^{q}d\w\Big)^{1/q}.
\]
This proves the second inequality in the first statement, and concludes the proof.
\end{proof}

\begin{remark}\label{rem:dyadicpar}
Let $f\in\Baspqq$ be such that 
\[
\supp(\wh{f}\,)\subseteq\{\xi\in\Rd\mid 2^{k-1}\leq |\xi|\leq 2^{k+1}, |\hat{\xi}-\nu|\leq 2^{1-k/2}\}
\]
for some $k\in \N$ and $\nu\in \Su^{d-1}$. Then Proposition \ref{prop:discrete} yields
\[
\|f\|_{\Baspqq}\eqsim 2^{k(s+\frac{d-1}{2}(\frac{1}{2}-\frac{1}{q}))}\|f\|_{L^{p}(\Rd)}\eqsim \|f\|_{B^{s+\frac{d-1}{2}(\frac{1}{2}-\frac{1}{q})}_{p,q}(\Rd)}.
\]
\end{remark}

\subsection{Embeddings}\label{subsec:embed}

We first obtain Sobolev embeddings into the Besov scale. Note that \eqref{eq:SobolevBesov} was already stated in \eqref{eq:SobolevBesovIntro} in the introduction, and recall the definition of $s(p)$ from \eqref{eq:sp}.

\begin{proposition}\label{prop:Sobolev}
Let $p\in[1,\infty)$ and $s\in\R$. Then
\begin{equation}\label{eq:SobolevBesov}
B^{s+s(p)}_{p,p}(\Rd)\subseteq\Basppp\subseteq B^{s-s(p)}_{p,p}(\Rd).
\end{equation}
Moreover, one has 
\begin{equation}\label{eq:SobolevBesov2}
B^{s+s(p)}_{p,p'}(\Rd)\subseteq \Ba^{s}_{p,p',p'}(\Rd),\quad 1< p\leq 2,
\end{equation}
and
\begin{equation}\label{eq:SobolevBesov3}
\Ba^{s}_{p,p',p'}(\Rd)\subseteq B^{s-s(p)}_{p,p'}(\Rd),\quad 2\leq p<\infty.
\end{equation}
\end{proposition}
\begin{proof}
By Proposition \ref{prop:discrete}, we may prove the required statements for $f\in\Sw'(\Rd)$ which satisfy $\supp(\wh{f}\,)\subseteq\{\xi\in\Rd\mid 2^{-1+k}\leq |\xi|\leq 2^{1+k}\}$ for some $k\in\N$. By recalling that the Sobolev and Besov norms of a function with frequency support in a dyadic annulus are equivalent, \eqref{eq:SobolevBesov} now follows directly from \eqref{eq:Sobolevintro}. 

Next, \eqref{eq:SobolevBesov2} and \eqref{eq:SobolevBesov3} follow from \eqref{eq:SobolevBesov} for $p=2$. Hence, by Proposition \ref{prop:discrete} and interpolation, it suffices to note that
\[
\max_{\nu\in\Theta_{k}}\|\chi_{\nu}(D)f\|_{L^{1}(\Rd)}\lesssim \|f\|_{L^{1}(\Rd)},
\]
since the kernels of the $\chi_{\nu}$ are uniformly in $L^{1}(\Rd)$, cf.~\eqref{eq:kernelbounds}, and that
\[
\|f\|_{L^{\infty}(\Rd)}=\Big\|\sum_{\nu\in\Theta_{k}}\chi_{\nu}(D)f\Big\|_{L^{\infty}(\Rd)}\leq \sum_{\nu\in\Theta_{k}}\|\chi_{\nu}(D)\|_{L^{\infty}(\Rd)}.\qedhere
\]
\end{proof}

\begin{remark}\label{rem:sharpness}
The Sobolev exponents in \eqref{eq:SobolevBesov} are sharp, by Proposition \ref{prop:PolynomialGrowth} and because the half-wave propagators lose $2s(p)$ derivatives in the Besov scale. Sharpness of one of the embeddings also follows from Remark \ref{rem:dyadicpar}.
\end{remark}

By Proposition \ref{prop:Sobolev} and standard embeddings for Besov spaces (see \cite[Section 2.3.2]{Triebel10}), one has
\[
\Ba^{s(p)}_{p,p,p}(\Rd) \subseteq L^p(\Rd) \subseteq \Ba^{-s(p)}_{p,p',p'}(\Rd)
\]
for $1<p\leq 2$, and 
\begin{equation}\label{eq:EmbeddingLargep}
\Ba^{s(p)}_{p,p',p'}(\Rd) \subseteq L^p(\Rd) \subseteq \Ba^{-s(p)}_{p,p,p}(\Rd)
\end{equation}
for $2\leq p<\infty$. These embeddings are similar to embeddings for modulation spaces.

Next, we obtain embeddings within the scales of adapted Besov spaces and Hardy spaces for FIOs. Combined with the Sobolev embeddings for $\Hp$ in \eqref{eq:Sobolevintro}, this proposition implies \eqref{eq:EmbeddingRefinement}.

\begin{proposition}\label{prop:Sobolev2}
Let $p,r\in[1,\infty]$, $q\in[1,\infty)$ and $s\in\R$. Then 
\begin{equation}\label{eq:SobolevBesov4}
\Ba^{s}_{p,q_{2},r}(\Rd)\subseteq \Ba^{s}_{p,q_{1},r}(\Rd)
\end{equation}
and
\begin{equation}\label{eq:SobolevBesov6}
\Ba^{s+\frac{d-1}{2}(\frac{1}{q_{1}}-\frac{1}{q_{2}})}_{p,q_{1},q_{1}}(\Rd)\subseteq \Ba^{s}_{p,q_{2},q_{2}}(\Rd)
\end{equation}
for all $q_{1},q_{2}\in[1,\infty)$ with $q_{1}\leq q_{2}$. Moreover, one has
\begin{equation}\label{eq:SobolevBesov8}
\Ba^{s+\frac{d+1}{2}(\frac{1}{p_{1}}-\frac{1}{p_{2}})}_{p_{1},q,q}(\Rd)\subseteq \Ba^{s}_{p_{2},q,q}(\Rd)
\end{equation}
for all $p_{1},p_{2}\in[1,\infty]$ with $p_{1}\leq p_{2}$, and for all $t\in[1,\infty]$ and $\veps>0$ one has
\begin{equation}\label{eq:SobolevBesov7}
\Ba^{s+\veps}_{p,q,r}(\Rd)\subseteq \Ba^{s}_{p,q,t}(\Rd)
\end{equation}
and
\begin{equation}\label{eq:SobolevBesov5}
\Ba^{s+\veps}_{p,p,r}(\Rd)\subseteq \Hps\subseteq \Ba^{s-\veps}_{p,p,r}(\Rd). 
\end{equation}
\end{proposition}
\begin{proof}
For \eqref{eq:SobolevBesov4} one can rely on H\"{o}lder's inequality, while \eqref{eq:SobolevBesov6} follows from Proposition \ref{prop:discrete} and the inclusion $\ell^{q_{1}}\subseteq \ell^{q_{2}}$.  Moreover, \eqref{eq:SobolevBesov7} and \eqref{eq:SobolevBesov5} follow from standard embeddings between Besov and Sobolev spaces (see \cite[Section 2.3.2]{Triebel10}). 

Finally, by Proposition \ref{prop:discrete}, for \eqref{eq:SobolevBesov8} it suffices to show that
\begin{equation}\label{eq:toshowSobolev}
\|\chi_{\nu}^{k}(D)f\|_{L^{p_{2}}(\Rd)}\lesssim 2^{k\frac{d+1}{2}(\frac{1}{p_{1}}-\frac{1}{p_{2}})}\|\chi^{k}_{\nu}(D)f\|_{L^{p_{1}}(\Rd)}
\end{equation}
for all $k\geq0$, $\nu\in\Theta_{k}$ and $f\in \Ba^{s}_{p_{1},q,q}(\Rd)$. To this end, let $\tilde{\chi}_{\nu}^{k}$ be as in \eqref{eq:packets}, with the natural modification for $k=0$. Then, as in \eqref{eq:kernelbounds}, integration by parts yields
\[
\|\F^{-1}(\tilde{\chi}_{\nu}^{k})\|_{L^{p_{3}}(\Rd)}\lesssim 2^{k\frac{d+1}{2}(\frac{1}{p_{1}}-\frac{1}{p_{2}})},
\]
where $\frac{1}{p_{3}}=1+\frac{1}{p_{2}}-\frac{1}{p_{1}}$. Now \eqref{eq:toshowSobolev} follows from Young's inequality, since $\chi_{\nu}^{k}=\tilde{\chi}_{\nu}^{k}\chi_{\nu}^{k}$.
\end{proof}

\section{Invariance and product estimates}\label{sec:invariance}

In this section we prove that $\Baspqq$ is invariant under the half-wave propagators and more general oscillatory Fourier multipliers. We also obtain some product estimates, which are useful for solving nonlinear equations.

\subsection{Invariance}\label{subsec:invariance}

The main result of this subsection is the following slightly more general version of \eqref{eq:PolynomialGrowthIntro}.

\begin{proposition}
\label{prop:PolynomialGrowth}
Let $\phi\in C^{\infty}(\Rd\setminus\{0\})$ be homogeneous of order $1$, and let $p\in[1,\infty]$, $q\in[1,\infty)$ and $s\in\R$. Then there exists a $C\geq0$ such that
\[
\| e^{it \phi(D)} f \|_{\Baspqq} \leq C (1+|t|)^{2s(p)} \| f \|_{\Baspqq}
\]
for all $t\in\R$ and $\Baspqq$.
\end{proposition}
\begin{proof}
By Proposition \ref{prop:discrete}, it suffices to show that 
\[
\|\chi_{\nu}^{k}(D)e^{it\phi(D)}f\|_{L^{p}(\Rd)}\lesssim (1+|t|)^{2s(p)}\|\chi_{\nu}^{k}(D)f\|_{L^{p}(\Rd)}
\]
for all $t\in\R$, $f\in\Baspqq$, $k\in\N_{0}$ and $\nu\in\Theta_{k}$. This is clearly true for $p=2$. Hence, by interpolation and duality, it suffices to show the statement for $p=1$. 

To do so, we will rely on a dilation argument. First note that 
\begin{equation}\label{eq:L1fixed}
\sup_{|t|\leq 4}\|\chi_{\nu}^{k}(D)e^{it\phi(D)}f\|_{L^{1}(\Rd)}\lesssim \|f\|_{L^{1}(\Rd)},
\end{equation}
as follows either from kernel bounds (see \cite[p.~406]{Stein93} and \cite[Remark 3.7]{Rozendaal2021}), or from the boundedness of $e^{it\phi(D)}$ on $\HT^{1}_{FIO}(\Rd)$ (see \cite[Corollary 3.6]{Rozendaal2021LocalSmoothing}), combined with \cite[Proposition 6.4]{FaLiRoSo19}. Either way, we may thus suppose that $|t|> 4$. 

Let $l\geq 2$ be such that $2^{l}<|t|\leq 2^{l+1}$. Then the dilated function $\chi_{\nu}^{k}(\frac{\cdot}{|t|})$ satisfies
\[
\supp(\chi_{\nu}^{k}(\tfrac{\cdot}{|t|}))\subseteq\{\xi\in\Rd\mid 2^{k+l-1}\leq |\xi|\leq 2^{k+l+2},|\hat{\xi}-\nu|\leq 2^{1-k/2}\}.
\]
Hence for each $\xi\in\Rd$ one has
\begin{equation}\label{eq:biggersupport}
\chi^{k}_{\nu}(\tfrac{\xi}{|t|})=\sum_{m=k+l-2}^{k+l+3}\sum_{\w\in\tilde{\Theta}_{k,l,\nu}}\psi_{m}(\xi)\chi_{\w}(\xi)\chi^{k}_{\nu}(\tfrac{\xi}{|t|}),
\end{equation}
where $\tilde{\Theta}_{k,l,\nu}:=\{\w\in \Theta_{k+l}\mid |\w-\nu|\leq 2^{2-k/2}\}$. Note that each $\tilde{\Theta}_{k,l,\nu}$ has approximately $2^{l(d-1)/2}\eqsim |t|^{(d-1)/2}$ elements. Moreover, as in \eqref{eq:L1fixed}, one has
\begin{equation}\label{eq:L1fixed2}
\|\psi_{m}(D)\chi_{\nu}(D)e^{i\phi(D)}g\|_{L^{1}(\Rd)}\lesssim \|g\|_{L^{1}(\Rd)}
\end{equation}
for all $k+l-2\leq m\leq k+l+3$, $\w\in \tilde{\Theta}_{k,l,\nu}$ and $g\in L^{1}(\Rd)$. Write $f_{t}(y):=|t|^{d}f(|t|y)$ for $y\in\Rd$. Then it suffices to combine \eqref{eq:biggersupport} and \eqref{eq:L1fixed2} with dilation arguments:
\begin{align*}
\|\chi_{\nu}^{k}(D)e^{it\phi(D)}f\|_{L^{1}(\Rd)}&=|t|^{-d}\big\|\big(\chi_{\nu}^{k}(\tfrac{D}{|t|})e^{i\phi(D)}f_{t}\big)\big(\tfrac{\cdot}{|t|}\big)\big\|_{L^{1}(\Rd)}\\
&=\|\chi_{\nu}^{k}(\tfrac{D}{|t|})e^{i\phi(D)}f_{t}\|_{L^{1}(\Rd)}\\
&\leq \sum_{m=k+l-2}^{k+l+3}\sum_{\w\in\tilde{\Theta}_{k,l,\nu}}\big\|\psi_{m}(D)\chi_{\w}(D)e^{i\phi(D)}\chi^{k}_{\nu}(\tfrac{D}{|t|})f_{t}\big\|_{L^{1}(\Rd)}\\
&\lesssim |t|^{(d-1)/2}\|\chi^{k}_{\nu}(\tfrac{D}{|t|})f_{t}\|_{L^{1}(\Rd)}=|t|^{(d-1)/2}\|\chi^{k}_{\nu}(D)f \|_{L^{1}(\Rd)},
\end{align*}
where we used in particular the homogeneity of $\phi$.
\end{proof}

In the case where $\phi(\xi)=|\xi|$, the growth rate in Proposition \ref{prop:PolynomialGrowth} is sharp. For $p=q$ this follows from \eqref{eq:SobolevBesov}, since one would otherwise obtain sharper bounds in the Besov scale than are known to be possible. Moreover, the following radial Knapp example shows sharpness of the polynomial growth bound for general $p$ and $q$, given that the $\Ba^{s}_{p,q,q}(\Rd)$ norm coincides with the $L^{p}(\Rd)$ norm for low frequencies. The analog for the Schr\"odinger equation was considered in \cite[Corollary~1.4]{Schippa2022}.

\begin{proposition}
\label{prop:SharpGrowth}
Let $p \in [1,\infty]$. Then there exist an $f \in \Sw(\R^d)$ with $\supp (\wh{f}\,) \subseteq \{\xi\in\Rd\mid |\xi|\leq 1\}$, and a $C\geq0$, such that
\[
\| e^{it \sqrt{-\Delta}} f \|_{L^p(\R^d)} \geq C (1+|t|)^{2s(p)} \| f \|_{L^p(\R^d)}
\]
for all $t\in\R$.
\end{proposition}
\begin{proof}
We can suppose that $t \gg 1$, and we firstly consider $p \in [1,2]$.

 Let $\chi \in C^\infty_c(\Rd)$ be not identically zero and such that $\chi(\xi)=0$ for $|\xi|\notin [1/2,1]$, and let $\chi_{0}\in C^{\infty}_{c}(\R)$ be such that $\chi(\xi) = \chi_0(|\xi|)$ for all $\xi\in\Rd$. We consider as initial data $f:=\F^{-1}(\chi)$, and rewrite the linear solution using radial symmetry:
\begin{equation*}
\begin{split}
(e^{it \sqrt{- \Delta}} f)(x) &= \int_{\R^d} e^{i(x\cdot\xi + t |\xi|)} \chi_0(|\xi|) d\xi \\
&= \int_0^\infty ds \int_{\mathbb{S}^{d-1}} d \sigma(\theta) e^{i(x\cdot s \theta + t s)} s^{d-1} \chi_0(s),
\end{split}
\end{equation*}
for $x\in\Rd$. We have
\begin{equation*}
\int_{\mathbb{S}^{d-1}} e^{i x\cdot\theta} d\theta = |x|^{-\frac{d-2}{2}} J_{\frac{d-2}{2}}(|x|)
\end{equation*}
with $J_\nu$ the Bessel function of the first kind:
\begin{equation*}
J_\nu(r) = \frac{(r/2)^\nu}{\Gamma(\nu + 1/2) \pi^{1/2}} \int_{-1}^1 e^{ir t} (1-t^2)^{\nu - \frac{1}{2}} dt, \quad \nu > -\frac{1}{2}.
\end{equation*}
We have the following asymptotic expansion by \cite[Section~10.17]{OlverMaximon2010}:
\begin{equation*}
J_\nu(z) = \big( \frac{2}{\pi z} \big)^{\frac{1}{2}} \big( \cos \omega \sum_{k=0}^\infty (-1)^k \frac{a_{2k}(\nu)}{z^{2k}} - \sin \omega \sum_{k=0}^\infty (-1)^{k} \frac{a_{2k+1}(\nu)}{z^{2k+1}} \big)
\end{equation*}
with $\omega = z - \frac{1}{2} \nu \pi - \frac{1}{4} \pi$. Hence, we can write
\begin{equation*}
\chi_0(s) (s|x|)^{-\frac{d-2}{2}} J_{\frac{d-2}{2}}(s |x|) = \sum_{j=0}^M |x|^{-\frac{d-1}{2}-j} \sum_{\pm} e^{\pm is |x|} \chi_{j,\pm}(s) + O(|x|^{-\frac{d-1}{2}-M})
\end{equation*}
with $\chi_{j,\pm} \in C^\infty_c(B(0,1) \backslash B(0,1/2))$. We find
\begin{equation*}
(e^{it \sqrt{- \Delta}} f)(x) = \sum_{j=0}^M |x|^{-\frac{d-1}{2}-j} \sum_{\pm} \int_\R e^{is t \pm i s |x|} \chi_{j,\pm}(s) ds + O(|x|^{-\frac{d-1}{2}-M}).
\end{equation*}
Let $t \gg 1$ and $\big| |x| - t \big| \leq 2^{-10}$. In this case $st + s |x|$ is a non-stationary phase, which means for $j=0$ the contribution of $\chi_{0,+}$ can be neglected against $\chi_{0,-}$. But by the explicit form of $\chi_{j,-}$, we have
\begin{equation*}
\big| \int e^{ist - is|x|} \chi_{0,-}(s) ds \big| \gtrsim 1.
\end{equation*}
Thus, the higher orders can likewise be neglected against the contribution of $\chi_{0,-}$. This shows $| e^{it \sqrt{- \Delta}} f|(x) \gtrsim |t|^{-\frac{d-1}{2}}$ for $\big| |x| - t \big| \leq 2^{-10}$, and this concludes the proof for $p \in [1,2]$ by integration. In the following let $\tilde{\rho}$ be a radially decreasing function with $\tilde{\rho}(\xi) = 1$ for $|\xi| \leq \frac{1}{2}$ and $\text{supp}(\tilde{\rho}) \subseteq B(0,3/4)$. Hence, we have proved in view of Proposition \ref{prop:PolynomialGrowth}
\begin{equation*}
\| e^{it \sqrt{- \Delta}} \tilde{\rho}(D) f \|_{L^p(\R^d)} \sim (1+|t|)^{2 s(p)} \text{ for } p \in [1,2].
\end{equation*}
For $p>2$ we argue by duality. The adjoint operator of $e^{it \sqrt{-\Delta}} \tilde{\rho}(D)$ is given by
\begin{equation*}
(e^{it \sqrt{-\Delta}} \tilde{\rho}(D))^{\text{ad}} = e^{-it \sqrt{-\Delta}} \tilde{\rho}(D),
\end{equation*}
and moreover, $c.c. ((e^{-it \sqrt{-\Delta}} \tilde{\rho}(D) \bar{g})) = e^{it \sqrt{-\Delta}} \tilde{\rho}(D) g$. This shows by density that there is $f \in \mathcal{S}(\R^d)$ with compactly supported Fourier transform such that
\begin{equation*}
\| e^{it \sqrt{-\Delta}} \tilde{\rho}(D) f \|_{L^p(\R^d)} \sim (1+|t|)^{2s(p)} \text{ for } 2 < p <\infty.
\end{equation*}
For $p=\infty$ we obtain from $\| e^{it \sqrt{-\Delta}} \tilde{\rho}(D) \|_{L^\infty \to L^\infty} \sim (1+|t|)^{2s(\infty)}$ a function $f \in L^\infty(\R^d) \cap C^\infty(\R^d)$ such that the above display holds for $p=\infty$. We use finite speed of propagation to argue that $\chi f$ for $\chi \in C^\infty_c(\R^d)$ still satisfies
\begin{equation*}
\| e^{it \sqrt{-\Delta}} \tilde{\rho}(D) (\chi f) \|_{L^p(\R^d)} \sim (1+|t|)^{2s(\infty)}.
\end{equation*}

\end{proof}

\subsection{Product estimates}\label{subsec:product}

We begin with a simple  bilinear estimate.

\begin{lemma}\label{lem:ProductEstimateI}
Let $p_{1},p_{2},p\in[1,\infty]$ be such that $\frac{1}{p}=\frac{1}{p_{1}}+\frac{1}{p_{2}}$, and let $s> \frac{3(d-1)}{4}$. Then there exists a $C\geq0$ such that, for all $f\in \Ba^{s}_{p_{1},1,1}(\Rd)$ and $g \in \mathcal{B}^s_{p_{2},1,1}(\R^d)$, one has $fg\in \Ba^{s}_{p,1,1}(\Rd)$ and
\[
\| f g \|_{\mathcal{B}^s_{p,1,1}(\R^d)} \leq C \| f \|_{\mathcal{B}^s_{p_1,1,1}(\R^d)} \| g \|_{\mathcal{B}^{s}_{p_2,1,1}(\R^d)}.
\]
\end{lemma}
\begin{proof}
We use paraproduct analysis. More precisely, one has
\begin{align*}
\|fg\|_{\Ba^{s}_{p,1,1}(\Rd)}&\eqsim\!\sum_{k=0}^{\infty}2^{k(s-\frac{d-1}{4})}\!\sum_{\nu\in\Theta_{k}}\!\Big\|\!\sum_{l,m=0}^{\infty}\sum_{\w\in\Theta_{l},\mu\in\Theta_{m}}\!\chi_{\nu}^{k}(D)(\chi_{\w}^{l}(D)f \!\cdot\!\chi_{\mu}^{m}(D)g)\Big\|_{L^{p}(\Rd)}\\
&\leq \!\sum_{k,l,m=0}^{\infty}2^{k(s-\frac{d-1}{4})}\!\sum_{\nu\in\Theta_{k},\w\in\Theta_{l},\mu\in\Theta_{m}}\|\chi_{\nu}^{k}(D)(\chi_{\w}^{l}(D)f \cdot \chi_{\mu}^{m}(D)g)\|_{L^{p}(\Rd)},
\end{align*}
by \eqref{eq:decompose} and Proposition \ref{prop:discrete}. We write the latter expression as $I_{1}+I_{2}+I_{3}$, where $I_{1}$ involves the sum over $m\leq  l-3$, $I_{2}$ the sum over $l-2\leq m\leq l+2$, and $I_{3}$ the sum over $m\geq l+3$. We will estimate each of these terms separately. In fact, by symmetry, it suffices to consider only $I_{1}$ and $I_{2}$.

For the $High \times Low$ term $I_{1}$, we only get a nonzero contribution if $l-3\leq k\leq l+3$, since the low-frequency factor $\chi_{\mu}^{m}(D)g$ does not essentially change the dyadic localization. However, it can change the angular localization. For $l\geq0$, $m\leq l-3$, $\w\in\Theta_{l}$ and $\mu\in\Theta_{m}$, we decompose the support of $\chi_{\mu}^{m}$, which is approximately a $2^{m/2} \times \ldots \times 2^{m/2} \times 2^{m}$ slab, into $2^{m/2} \times \ldots \times 2^{m/2} \times 2^{\min(l/2,m)}$ slabs. Let $(\chi_\mu^{m,i})_{i\in I}$ be a corresponding partition of unity, with $|I| \eqsim  1+2^{m-l/2}$. Then the support of the convolution of $\chi_{\w}^{l}$ with a given $\chi_{\mu}^{m,i}$ can only intersect the support of $O(1)$ elements of $\Theta_{k}$. Hence the support of the convolution of $\chi_{\w}^{\l}$ and $\chi_{\mu}^{m}$ can only intersect the support of
$O(1+2^{m-l/2})$ elements of $\Theta_{k}$. Since $m-l/2\leq m/2$, we can combine \eqref{eq:kernelbounds} and H\"{o}lder's inequality to obtain
\begin{align*}
I_{1}&\eqsim \sum_{j=-3}^{3}\sum_{l=0}^{\infty}\sum_{m=0}^{l-3}2^{l(s-\frac{d-1}{4})}\sum_{\nu\in\Theta_{l+j},\w\in\Theta_{l},\mu\in\Theta_{m}}\|\chi_{\nu}^{l+j}(D)(\chi_{\w}^{l}(D)f \cdot \chi_{\mu}^{m}(D)g)\|_{L^{p}(\Rd)}\\
&\lesssim \sum_{l=0}^{\infty}\sum_{m=0}^{l-3}2^{l(s-\frac{d-1}{4})+\frac{m}{2}}\sum_{\w\in\Theta_{l},\mu\in\Theta_{m}}\|\chi_{\w}^{l}(D)f\|_{L^{p_{1}}(\Rd)}\|\chi_{\mu}^{m}(D)g\|_{L^{p_{2}}(\Rd)}\\
&\lesssim \|f\|_{\Ba^{s}_{p_{1},1,1}(\Rd)}\|g\|_{\Ba^{s}_{p_{2},1,1}(\Rd)},
\end{align*}
where in the final step we used Proposition \ref{prop:discrete} and that $s\geq (d+1)/4$.

For the $High \times High$ term $I_{2}$, all information on angular localization is lost. By trivially summing over $\nu\in\Theta_{k}$, using also \eqref{eq:kernelbounds}, H\"{o}lder's inequality and that $s >3(d-1)/4$, we obtain
\begin{align*}
I_{2}&=\sum_{l=0}^{\infty}\sum_{m=l-2}^{l+2}\sum_{k=0}^{l+5} 2^{k(s- \frac{d-1}{4})}\sum_{\nu\in\Theta_{k},\w\in\Theta_{l},\mu\in\Theta_{m}}\|\chi_{\nu}^{k}(D)(\chi_{\w}^{l}(D)f \cdot \chi_{\mu}^{m}(D)g)\|_{L^p(\Rd)}\\
&\lesssim \sum_{l=0}^{\infty}\sum_{m=l-2}^{l+2}\sum_{k=0}^{l+5} 2^{k(s+\frac{d-1}{4})} \sum_{\w\in\Theta_{l},\mu\in\Theta_{m}}\|\chi_{\w}^{l}(D)f\|_{L^{p_{1}}(\Rd)}\|\chi_{\mu}^{m}(D)g\|_{L^{p_{2}}(\Rd)}\\
&\lesssim \|f\|_{\Ba^{s}_{p_{1},1,1}(\Rd)}\|g\|_{\Ba^{s}_{p_{2},1,1}(\Rd)}.\qedhere
\end{align*}
\end{proof}

A trilinear estimate can be proved by similar means.

\begin{lemma}\label{lem:TrilinearEstimate}
Let $p_{1},p_{2},p_{3},p\in[1,\infty]$ be such that $\frac{1}{p}=\sum_{i=1}^{3}\frac{1}{p_{i}}$, and let $s>\frac{3(d-1)}{4}-1$ be such that $s\geq \frac{d-1}{4}$. Then there exists a $C\geq0$ such that, for all $f_{i}\in \Ba^{s}_{p_{i},1,1}(\Rd)$, $1\leq i\leq 3$, one has $\prod_{i=1}^{3}f_{i}\in \Ba^{s}_{p,1,1}(\Rd)$ and
\[
\Big\| \prod_{i=1}^3 f_i \Big\|_{\mathcal{B}^{s-1}_{p,1,1}(\R^d)} \leq C\prod_{i=1}^3 \| f_i \|_{\mathcal{B}^s_{p_i,1,1}(\R^d)}.
\]
\end{lemma}
\begin{proof}
The approach to the proof is similar to Lemma \ref{lem:ProductEstimateI}, so we only indicate how to deal with the relevant terms, involving indices $k,k_{i}\in\N_{0}$ for $1\leq i\leq 3$. 

For the $High \times Low \times Low$ term, we consider $2^k \eqsim 2^{k_1} \gg 2^{k_{2}} \geq 2^{k_3}$ and
\[
I:=2^{k(s-1-\frac{d-1}{4})}\sum_{\nu\in\Theta_{k}}\Big\| \chi^N_\nu(D) \Big(\prod_{i=1}^{3}\chi_{\nu_{i}} ^{k_{i}}(D)f_{i}\Big)\Big\|_{L^{p}(\Rd)},
\]
for $\nu_{i}\in\Theta_{k_{i}}$, $1\leq i\leq 3$. We have to estimate the number of $\nu$ for which the support of $\chi_{\nu}$ intersects the support of $\chi^{k_1}_{\nu_{1}} * \chi^{k_2}_{\nu_2} * \chi^{k_3}_{\nu_3}$. Note that $\chi^{k_2}_{\nu_2}* \chi^{k_3}_{\nu_3}$ is supported in a slab of dimensions approximately $2^{\max(k_3,k_2/2)} \times 2^{k_2/2}\times \ldots \times 2^{k_2/2} \times 2^{k_2}$. This we subdivide into cubes of side length no more than $2^{k_1/2}$, of which there are no more than approximately $(1+2^{k_3-k_{1}/2})(1+ 2^{k_2-k_{1}/2})$. This yields
\[
I\lesssim 2^{k_{1}(s-1-\frac{d-1}{4})} (1+2^{k_3/2})(1+ 2^{k_2/2}) \prod_{i=1}^3  \| \chi^{k_i}_{\nu_i}(D) f_i \|_{L^{p_i}(\Rd)}.
\]
Since $s\geq \frac{d-1}{4}$, this suffices for the $High \times Low \times Low$ term.

Next, for the $High \times High \times Low$ term, we consider $2^{k_1} \eqsim 2^{k_2} \gg 2^{k_3}$. Then nonzero contributions only arise for $2^{k}\lesssim 2^{k_{1}}$. Moreover, trivial summation yields
\begin{align*}
&\;2^{k(s-1-\frac{d-1}{4})}\sum_{\nu_{k}\in\Theta_{k}} \Big\| \chi^k_\nu(D) \Big(\prod_{i=1}^3 \chi^{k_i}_{\nu_i}(D) f_i \Big) \Big\|_{L^p(\Rd)}\\
& \lesssim 2^{k(s-1+\frac{d-1}{4})} \prod_{i=1}^3 \| \chi^{N_i}_{\nu_i}(D) f_i \|_{L^{p_i}(\Rd)}.
\end{align*}
Thus, to deal with the $High \times High \times Low$ term, it suffices to show that
\[
k(s-1+\tfrac{d-1}{4}+\veps)\leq   (k_{1}+k_{2}+k_{3})(s-\tfrac{d-1}{4})+M
\]
for some $\veps,M>0$. The above display is trivial for $s-1 + \frac{d-1}{4} < 0$, so we assume this is not the case. We first use that $2^{k}\lesssim 2^{k_{1}}\eqsim 2^{k_{2}}$ and $s>\frac{3(d-1)}{4}-1$ to find:
\begin{align*}
k(s-1+\tfrac{d-1}{4}+\veps)&\leq \tfrac{k_{1}+k_{2}}{2}(s-1+\tfrac{d-1}{4}+\veps)+M\\
&\leq (k_{1}+k_{2})(s-\tfrac{d-1}{4})+M,
\end{align*}
for suitable $\veps,M>0$. Moreover, $0\leq k_{3}(s-\frac{d-1}{4})$ since $s\geq \frac{d-1}{4}$.

The $High \times High \times High$ term can be dealt with in the exact same way as the $High \times High \times Low$ term. By symmetry, this concludes the proof.
\end{proof}

\section{Local smoothing in $\Ba^{s}_{p,2,2}(\Rd)$}
\label{section:LocalSmoothing}

In this section we prove Theorem \ref{thm:LocalSmoothingHardy}. The proof is analogous to that of \cite[Theorem 1.1]{Rozendaal2021LocalSmoothing}. In particular, the key to the proof is the following proposition, which generalizes the case $q=p$ in \cite[Corollary 4.2]{Rozendaal2021LocalSmoothing} to arbitrary $q\in[1,\infty)$.

\begin{proposition}\label{prop:equivwave}
Let $p\in[1,\infty]$, $q\in[1,\infty)$ and $s \in \R$, and let $0\neq g\in\Sw(\R)$. Then there exists a $C>0$ such that the following holds. Let $f\in\Sw'(\Rd)$ be such that $\supp(\wh{f}\,)\subseteq\{\xi\in\Rd\mid 2^{k-1}\leq |\xi|\leq 2^{k+1}\}$ for some $k\in\N$. Then
\begin{align*}
\frac{1}{C}\|f\|_{\Baspqq}&\leq 2^{k(s+\frac{d-1}{2}(\frac{1}{2}-\frac{1}{q}))}\Big(\sum_{\nu\in\Theta_{k}}\|g(t)e^{it\sqrt{-\Delta}}\chi_{\nu}(D)f\|_{L^{p}(\R\times\Rd)}^{q}\Big)^{1/q}\\
&\leq C\|f\|_{\Baspqq}
\end{align*} 
whenever one of these quantities is finite. Hence an $f \in \Sw'(\Rd)$ satisfies $f\in\Baspqq$ if and only if 
\begin{equation}\label{eq:equivalent2}
\Big(\sum_{k=0}^{\infty}2^{qk(s+\frac{d-1}{2}( \frac{1}{2} - \frac{1}{q}))}\sum_{\nu\in\Theta_{k}}\| g(t)e^{it\sqrt{-\Delta}}\chi_\nu^{k}(D)f \|^q_{L^{p}(\R\times \Rd)}\Big)^{1/q}
\end{equation}
is finite, and \eqref{eq:equivalent2} defines an equivalent norm on $\Baspqq$.
\end{proposition}
\begin{proof}
It is straightforward to deal with the low frequencies, using similar estimates as in the high-frequency case, so we may assume that $\rho(D)f=0$. Moreover, by Lemma \ref{lem:dyadic}, the second statement follows from the first.

For the first statement, note that there exists an $N\geq0$ such that
\[
\|e^{it\sqrt{-\Delta}}\chi_{\nu}(D)f\|_{L^{p}(\Rd)}\lesssim (1+|t|)^{N}\|\chi_{\nu}(D)f\|_{L^{p}(\Rd)}
\]
for all $\nu\in\Theta_{k}$ and $t\in\R$, as follows either from kernel bounds, or by combining Remark \ref{rem:dyadicpar} and Proposition \ref{prop:PolynomialGrowth}. Either way, one thus has
\begin{align*}
\|\chi_{\nu}(D)g(t)e^{it\sqrt{-\Delta}}f\|_{L^{p}(\R\times\Rd)}&=\Big(\int_{\R}|g(t)|\|e^{it\sqrt{-\Delta}}\chi_{\nu}(D)f\|^{p}_{L^{p}(\Rd)}dt\Big)^{1/p}\\
&\lesssim \|\chi_{\nu}(D)f\|_{L^{p}(\Rd)}.
\end{align*}
This in turn yields
\begin{align*}
&2^{k(s+\frac{d-1}{2}(\frac{1}{2}-\frac{1}{q}))}\Big(\sum_{\nu\in\Theta_{k}}\|g(t)e^{it\sqrt{-\Delta}}\chi_{\nu}(D)f\|_{L^{p}(\R\times\Rd)}^{q}\Big)^{1/q}\\
&\lesssim 2^{k(s+\frac{d-1}{2}(\frac{1}{2}-\frac{1}{q}))}\Big(\sum_{\nu\in\Theta_{k}}\|\chi_{\nu}(D)f\|_{L^{p}(\Rd)}^{q}\Big)^{1/q}\eqsim \|f\|_{\Baspqq},
\end{align*}
by Proposition \ref{prop:discrete}.

On the other hand, for all $\nu\in\Theta_{k}$ one has
\[
\|\chi_{\nu}(D)f\|_{L^{p}(\Rd)}=\|e^{-it\sqrt{-\Delta}}\chi_{\nu}(D)e^{it\sqrt{-\Delta}}f\|_{L^{p}(\Rd)}\lesssim \|g(t)e^{it\sqrt{-\Delta}}\chi_{\nu}(D)f\|_{L^{p}(\Rd)}
\]
on any compact interval $I\subseteq\R$ such that $|g(t)|\gtrsim 1$ for all $t\in I$. Hence
\begin{align*}
\|f\|_{\Baspqq}&\eqsim 2^{k(s+\frac{d-1}{2}(\frac{1}{2}-\frac{1}{q}))}\Big(\sum_{\nu\in\Theta_{k}}\|\chi_{\nu}(D)f\|_{L^{p}(\Rd)}^{q}\Big)^{1/q}\\
&\lesssim 2^{k(s+\frac{d-1}{2}(\frac{1}{2}-\frac{1}{q}))}\Big(\sum_{\nu\in\Theta_{k}}\|g(t)e^{it\sqrt{-\Delta}}\chi_{\nu}(D)f\|_{L^{p}(\R\times\Rd)}^{q}\Big)^{1/q},
\end{align*}
again by Proposition \ref{prop:discrete}.
\end{proof}

The proof of Theorem \ref{thm:LocalSmoothingHardy} is now almost immediate.

\begin{proof}[Proof~of~Theorem~\ref{thm:LocalSmoothingHardy}]
Let $g\in\Sw(\R)$ be such that $|g(t)|\geq1$ for $t\in[0,1]$, and $\supp(\wh{g}\,)\subseteq [-1,1]$. Let $\veps>0$ and $f\in\Ba^{s}_{p,2,2}(\Rd)$. We apply the Littlewood--Paley decomposition $(\psi_{k})_{k=0}^{\infty}$ to $f$. Moreover, we can use a kernel estimate for the low frequencies, so we may assume that $\psi_{0}(D)f=0$. Then the $\ell^{2}$-decoupling inequality \eqref{eq:DecouplingConeIntroduction}, with $\veps$ replaced by $\veps/2$, yields
\begin{align*}
&\| e^{it \sqrt{- \Delta}} f \|_{L_t^p([0,1],L^p(\R^d))} \leq  \sum_{k=1}^{\infty}\|e^{it \sqrt{- \Delta}} \psi_{k}(D)f \|_{L_t^p([0,1],L^p(\R^d))}\\
&\lesssim \sum_{k=1}^{\infty}2^{k(\overline{s}(p) + \varepsilon/2)} \Big( \sum_{\nu\in\Theta_{k}} \|g(t)e^{it \sqrt{- \Delta}} \chi_\nu(D)\psi_{k}(D)f \|^2_{L^p(\R\times\Rd)} \Big)^{1/2}.
\end{align*}
Now Proposition \ref{prop:equivwave} implies that the final quantity is equivalent to 
\begin{align*}
\sum_{k=1}^{\infty}2^{-k\varepsilon/2} \|\psi_{k}(D)f\|_{\Ba^{\overline{s}(p)+\veps}_{p,2,2}(\Rd)}\lesssim \sum_{k=1}^{\infty}2^{-k\varepsilon/2} \|f\|_{\Ba^{\overline{s}(p)+\veps}_{p,2,2}(\Rd)}\eqsim \|f\|_{\Ba^{\overline{s}(p)+\veps}_{p,2,2}(\Rd)}.
\end{align*}
This concludes the proof.
\end{proof}

\begin{remark}\label{rem:sharpness2}
For each $2<p<\infty$ the exponent $\overline{s}(p)$ in Theorem \ref{thm:LocalSmoothingHardy} is sharp, in the sense that, for any $s<\overline{s}(p)$, there does not exist a $C\geq0$ such that
\begin{equation}\label{eq:sharpness}
\| e^{it \sqrt{- \Delta}} f \|_{L_t^p([0,1],L^p(\R^d))}\leq C\|f\|_{\Ba^{s}_{p,2,2}(\Rd)}
\end{equation}
for all $f\in\Ba^{s}_{p,2,2}(\Rd)$. This follows immediately from the sharpness of the estimates in \cite{Rozendaal2021LocalSmoothing}, combined with \eqref{eq:EmbeddingRefinement}. However, we can also give a more direct argument.

Indeed, first note that \eqref{eq:sharpness} and \eqref{eq:EmbeddingRefinement} combine to yield
\begin{equation}\label{eq:sharpness2}
\| e^{it \sqrt{- \Delta}} f \|_{L_t^p([0,1],L^p(\R^d))}\lesssim_{\veps}\|f\|_{W^{s+s(p)+2\veps,p}(\Rd)}
\end{equation}
for all $f\in W^{s+s(p)+2\veps,p}(\Rd)$ and $\veps>0$. Hence, for $p\geq 2(d+1)/(d-1)$, by choosing $\veps$ sufficiently small, \eqref{eq:sharpness2} improves upon the local smoothing estimates in \eqref{eq:LocalSmoothingConjecture}. Since these are known to be sharp, \eqref{eq:sharpness} cannot hold for $p\geq 2(d+1)/(d-1)$.

On the other hand, suppose $f\in \Ba^{s}_{p,2,2}(\Rd)$ is such that 
\[
\supp(\wh{f}\,)\subseteq\{\xi\in\Rd\mid 2^{k-1}\leq |\xi|\leq 2^{k+1}, |\hat{\xi}-\nu|\leq 2^{1-k/2}\}
\]
for some $k\in\N$ and $\nu\in\Su^{d-1}$. Then Remark \ref{rem:dyadicpar}, Proposition \ref{prop:PolynomialGrowth} and \eqref{eq:sharpness} yield
\begin{align*}
\|f\|_{L^{p}(\Rd)}&\eqsim \|f\|_{B^{0}_{p,2}(\Rd)}\eqsim \|f\|_{\Ba^{0}_{p,2,2}(\Rd)}\eqsim \| e^{it \sqrt{- \Delta}} f \|_{L_t^p([0,1],\Ba^{0}_{p,2,2}(\R^d))}\\
&\eqsim \| e^{it \sqrt{- \Delta}} f \|_{L_t^p([0,1],L^p(\R^d))}\lesssim \|f\|_{\Ba^{s}_{p,2,2}(\Rd)}\eqsim \|f\|_{B^{s}_{p,2}(\Rd)}\eqsim 2^{ks}\|f\|_{L^{p}(\Rd)}.
\end{align*}
Since $\overline{s}(p)=0$ for $2<p<2(d+1)/(d-1)$, this leads to a contradiction, and \eqref{eq:sharpness} cannot hold for such $p$.
\end{remark}

\section{Well-posedness for nonlinear wave equations}
\label{section:LWPNLW}

In this section we will mainly focus on the cubic nonlinear wave equation
\begin{equation}
\label{eq:CubicNLW}
\begin{cases}
\partial_t^2 u - \Delta_{x} u = \pm |u|^2 u, \quad (t,x) \in \R \times \R^2, \\
u(0) = f \in X, \quad \dot{u}(0) = g \in Y,
\end{cases}
\end{equation}
outside $L^2$-based Sobolev spaces. 

We first collect some preliminaries. In Section \ref{subsec:localwell} we then prove local well-posedness results for slowly decaying initial data, including a theorem for the quintic nonlinear wave equation. The local results do not distinguish between focusing and defocusing nonlinearities. In Section \ref{subsec:slower} we prove local results for initial data which decay even slower than in Section \ref{subsec:localwell}, and finally we prove global results for the \emph{defocusing} equation, that is, \eqref{eq:CubicNLW} with a minus sign on the right hand-side.

\subsection{Preliminaries}\label{subsec:prelim}

Our notion of well-posedness is based on \cite[Section~3]{BejenaruTao2006}. We recall the key elements. We use Duhamel's formula to write \eqref{eq:CubicNLW} as an abstract evolution equation:
\begin{equation}
\label{eq:AbstractEvolution}
u = L(f,g) + N_3(u,u,u),
\end{equation}
where $u \in S$, which is a space-time function space, $L: X \times Y \to S$ is a densely defined linear operator, and $N_3:S \times S \times S \to S$ is a densely defined operator which is either linear or antilinear in each of its variables. In our case one has
\begin{equation*}
\begin{split}
L(f,g)(t)&:=\cos(t\sqrt{-\Delta})f+\frac{\sin(t\sqrt{-\Delta})}{\sqrt{-\Delta}}g, \\
N_{3}(u_{1},u_{2},u_{3})(t)&:=\pm \int_{0}^{t}\frac{\sin((t-s)\sqrt{-\Delta})}{\sqrt{-\Delta}}u_{1}(s)\overline{u_{2}(s)}u_{3}(s)ds.
\end{split}
\end{equation*}
 We say that \eqref{eq:AbstractEvolution} is \emph{quantitatively well posed} (with initial data space $X \times Y$ and solution space $S$) if there exists a $C\geq0$ such that
\begin{align}
\label{eq:LinearAbstractEstimate}
\| L(f,g) \|_{S} &\leq C \| (f,g) \|_{X \times Y}, \\
\label{eq:NonLinearAbstractEstimate}
\| N_3(u_1,u_2,u_3) \|_S &\leq C \prod_{i=1}^3 \| u_i \|_S,
\end{align}
for all $(f,g) \in X \times Y$ and $u_i \in S$, $1\leq i\leq 3$. 

If \eqref{eq:AbstractEvolution} is quantitatively well posed, then it follows from a fixed-point argument (see \cite[Theorem 3]{BejenaruTao2006}) that \eqref{eq:AbstractEvolution} is \emph{analytically locally well posed}. In particular, there exist $C_{0},\veps_{0}>0$ such that, for all $(f,g)\in B_{(X,Y)}(0,\veps_{0})=\{(f',g')\in X\times Y\mid \|(f',g')\|_{X\times Y}<\veps_{0}\}$, there exists a unique solution $u[f,g]\in B_{S}(0,C_{0}\veps_{0})$ to \eqref{eq:AbstractEvolution}. Moreover, the map $(f,g)\mapsto u[f,g]$ is Lipschitz continuous from $B_{(X,Y)}(0,\veps_{0})$ to $B_{S}(0,C_{0}\veps_{0})$, and one can expand $u[f,g]$ in terms of its Picard iterates. That is, define the nonlinear maps $A_m: X \times Y \to S$ recursively:
\begin{equation}\label{eq:Am}
\begin{split}
A_1(f,g) &:= L(f,g), \\
A_m(f,g) &:= \sum_{\substack{ m_1,m_2,m_3 \geq 1: \\ m_1 + m_2 + m_3 = m }} N_3(A_{m_1}(f,g),A_{m_2}(f,g),A_{m_3}(f,g))\quad \text{for } m > 1. 
\end{split}
\end{equation}
Then
\[
u[f,g]=\sum_{m=0}^{\infty}A_{m}(f,g),
\]
where the series converges absolutely in $S$ for all $(f,g)\in B_{(X,Y)}(0,\veps_{0})$. In what follows we define solution spaces $S_{T}$ locally in time, and by improving the estimate \eqref{eq:NonLinearAbstractEstimate} to
\begin{equation}
\label{eq:AbstractEstimatesMod}
\begin{split}
\| L(f,g) \|_{S_T} &\leq C \| (f,g) \|_{X \times Y}, \\
\| N_3(u_1,u_2,u_3) \|_{S_T} &\leq C T^\delta \prod_{i=1}^3 \| u_i \|_{S_T},
\end{split}
\end{equation}
for some $\delta > 0$, we can find a $T=T(\| (f,g) \|_{X \times Y})$, also for large data, such that analytic dependence on the initial data holds in $S_T$. Note that the additional gain in powers of $T$ is only required in the nonlinear estimate, which must be controlled in the Picard iteration.

\medskip

We use the following sharp local smoothing estimate due to Guth--Wang--Zhang \cite{GuthWangZhang2020} to prove the linear estimate \eqref{eq:LinearAbstractEstimate} for initial data in $L^p$-based Sobolev spaces.

\begin{theorem}\label{thm:localsmooth}
Let $p\in(2,\infty)$ and $s > \max( \frac{1}{2} - \frac{2}{p}, 0)$. Then there exists a $C\geq0$ such that
\[
\| e^{it \sqrt{- \Delta}} f \|_{L_t^p([0,1],L^p(\R^2))} \leq C \| f \|_{W^{s,p}(\R^2)}
\]
for all $f\in W^{s,p}(\R^{2})$.
\end{theorem}

The smoothing estimate for data in $\mathcal{B}^{s}_{p,2,2}(\R^2)$ is provided by Theorem \ref{thm:LocalSmoothingHardy}.  For the proof of the nonlinear estimate \eqref{eq:NonLinearAbstractEstimate}, we use Strichartz estimates (cf. \cite{KeelTao1998}).
 
\begin{theorem}\label{thm:StrichartzEstimates}
For $i\in\{1,2\}$, let $p_i, q_i \in[2,\infty]$, $q_i \neq \infty$, and $s_{i}\in\R$ be such that $\frac{2}{p_i} + \frac{1}{q_i} = \frac{1}{2}$ and $s_i = 2( \frac{1}{2} - \frac{1}{q_i} ) - \frac{1}{p_i}$. Then there exists a $C\geq0$ such that
for $u = e^{it \sqrt{- \Delta}} u_0 + \int_0^t e^{i(t-s) \sqrt{-\Delta}} f(s) ds$ the following estimate holds:
\begin{equation*}
\| \langle D \rangle^{-s_1} u \|_{L_t^{p_1}([0,T],L^{q_1}(\R^d))} \leq C (\| u_0 \|_{L^2(\R^d)} + \| \langle D \rangle^{s_2} f \|_{L_t^{p_2'}([0,T],L^{q_2'}(\R^d))}.
\end{equation*}
\end{theorem}

\subsection{Local well-posedness results}\label{subsec:localwell}

We begin with the local well-posedness result in Theorem \ref{thm:LWPNLW} and Remark \ref{rem:LWP}. 

\begin{proof}[Proof~of~Theorem~\ref{thm:LWPNLW}~and~Remark~\ref{rem:LWP}]
As explained above, it suffices to show that \eqref{eq:AbstractEstimatesMod} holds. In what follows, let $T \leq 1$, which simplifies powers of $T$. More precisely, the proof for $T>1$ is essentially identical, but one has to take into account that low-frequency terms contribute bounds which depend on a different power of $T$. 

We first consider the linear estimate \eqref{eq:LinearAbstractEstimate} with initial data in $W^{s,p}(\Rd)$, as in Remark \ref{rem:LWP}. Theorem \ref{thm:localsmooth} and H\"{o}lder's inequality in time yield
\begin{align*}
\| e^{it \sqrt{- \Delta}} f \|_{L_t^{24/7}([0,T],L^4(\R^2))} &\lesssim T^{1/24} \| f \|_{W^{\veps,4}(\R^2)},\\
\| e^{it \sqrt{- \Delta}} f \|_{L_t^4([0,T],L^6(\R^2))} &\lesssim T^{1/12} \| f \|_{W^{1/6+\varepsilon,6}(\R^2)}.
\end{align*}
This yields the linear estimate for $\cos(t \sqrt{- \Delta}) f$ and for the high frequencies of $\frac{\sin (t \sqrt{- \Delta})}{\sqrt{- \Delta}}g$. On the other hand, the low-frequency estimate holds since
\begin{equation}\label{eq:Mikhlin}
\rho(D) \lb D\rb \frac{\sin (t \sqrt{- \Delta})}{\sqrt{- \Delta}}: L^p(\R^{2}) \to L^p(\R^{2})
\end{equation}
for all $1<p<\infty$, with locally uniform bounds in $t$, due to Mikhlin's theorem. Here $\rho\in C^{\infty}_{c}(\R^{2})$ is the low-frequency cutoff from before. Note that for $|t| \gg 1$ we had to take into account growth in $t$.

Now consider the linear estimate \eqref{eq:LinearAbstractEstimate} for initial data in $\mathcal{B}^{s}_{p,2,2}(\R^2)+\dot{H}^{\tilde{s}}(\R^{2})$, cf.~Theorem \ref{thm:LWPNLW}. Recall that, for $p=4$, we consider the solution space
\begin{equation}\label{eq:solutionspace}
S_{T}=L^{24/7}_t([0,T],L^{4}(\Rtwo))\cap C([0,T],\Ba^{\veps}_{4,2,2}(\Rtwo)+\dot{H}^{3/8}(\Rtwo)).
\end{equation}
To obtain the linear estimate for the first space on the right-hand side, we again rely on H\"older's inequality, Theorem \ref{thm:LocalSmoothingHardy}, and on linear Strichartz estimates as in Theorem \ref{thm:StrichartzEstimates}. More precisely, let $f = f_1 + f_2$ with $f_1 \in \mathcal{B}^{\varepsilon}_{4,2,2}(\R^2)$ and $f_2 \in \dot{H}^{3/8}(\R^2)$. Then Theorems \ref{thm:LocalSmoothingHardy} and \ref{thm:StrichartzEstimates} yield
\begin{equation*}
\| e^{it \sqrt{-\Delta}} f \|_{L_t^{24/7}([0,T], L^4(\R^{2}))} \lesssim T^{1/24} (\| f_1 \|_{\mathcal{B}^{\varepsilon}_{4,2,2}(\R^2)} + \| f_2 \|_{\dot{H}^{3/8}(\R^2)}).
\end{equation*}
Note that we can likewise estimate $f_2$ in $H^s(\Rtwo)$ for $s \geq 3/8$. By taking the infimum over all decompositions $f = f_1+f_2$ in $\mathcal{B}^\varepsilon_{4,2,2}(\Rtwo) + \dot{H}^{3/8}(\Rtwo)$, we find
\begin{equation*}
\| \cos(t \sqrt{-\Delta}) f \|_{L_t^{24/7}([0,T], L^4(\R^{2}))} \lesssim T^{1/24} \| f \|_{\mathcal{B}^{\varepsilon}_{4,2,2}(\R^2) + \dot{H}^{3/8}(\R^2)}.
\end{equation*}
Next, write $g=g_{1}+g_{2}$ with $g_{1}\in B^{\veps-1}_{4,2,2}(\Rtwo)$ and $g_{2}\in \dot{H}^{-5/8}(\Rtwo)$. To obtain
\begin{equation*}
\Big\| \frac{\sin(t \sqrt{-\Delta})}{\sqrt{-\Delta}} g \Big\|_{L_t^{24/7}([0,T], L^4(\R^{2}))} \lesssim T^{1/24} (\| g_{1} \|_{\mathcal{B}^{\varepsilon-1}_{4,2,2}(\R^2)} + \|g_{2}\|_{\dot{H}^{-5/8}(\R^2)})
\end{equation*}
one proceeds in the same way when it comes to $g_{2}$ and the high frequencies of $g_{1}$, using the additional smoothing. On the other hand, for the low frequencies of $g_{1}$, one can argue as in \eqref{eq:Mikhlin}. Indeed, one has
\begin{equation}\label{eq:lowfreqadapted}
\rho(D) \frac{\sin (t \sqrt{- \Delta})}{\sqrt{- \Delta}}: \Ba^{\veps-1}_{p,2,2}(\R^{2}) \to\Ba^{s(p)+\veps}_{p,p,p}(\Rtwo)\subseteq B^{\veps}_{p,p}(\Rtwo)\subseteq L^p(\R^{2}).
\end{equation}
Here we used Proposition \ref{prop:discrete}, Mikhlin's theorem and trivial summation to obtain the mapping property, and \eqref{eq:SobolevBesov} and standard embeddings from Besov spaces into $L^{p}(\Rd)$ for the inclusions. This proves the linear estimate \eqref{eq:LinearAbstractEstimate} for the first space on the right-hand side of \eqref{eq:solutionspace}.

To show the linear estimate involving the solution space $C([0,T],\Ba^{\veps}_{4,2,2}(\Rtwo)+\dot{H}^{3/8}(\Rtwo))$, we use the invariance of $\Ba^{\veps}_{4,2,2}(\Rtwo)$ and $\dot{H}^{3/8}(\Rtwo))$ under the half-wave group, and a similar argument as in \eqref{eq:lowfreqadapted} to deal with the low frequencies of $\sin(t\sqrt{-\Delta})/\sqrt{-\Delta}$.

Finally, by relying instead on the $L^{6}_{t}([0,T],L^{6}(\Rtwo))$ smoothing estimate in Theorem \ref{thm:LocalSmoothingHardy}, as well as the $L_t^6([0,T], L^6(\R^2))$ Strichartz estimate, we obtain
\begin{equation*}
\| \cos(t \sqrt{-\Delta}) f \|_{L_t^4([0,T],L^6(\R^2))} \lesssim T^{1/12} \| f \|_{\mathcal{B}^\varepsilon_{6,2,2}(\R^2) + \dot{H}^{1/2}(\R^2)}.
\end{equation*}
Similarly,
\begin{equation*}
\Big\| \frac{\sin(t \sqrt{-\Delta})}{\sqrt{-\Delta}} g \Big\|_{L_t^4([0,T], L^6(\R^{2}))} \lesssim T^{1/12} \| g \|_{\mathcal{B}^{\varepsilon-1}_{6,2,2}(\R^2) + \dot{H}^{-1/2}(\R^2)}.
\end{equation*}
Moreover, to obtain the linear estimate for the solution space $C([0,T];\Ba^{\veps}_{6,2,2}(\Rtwo)+\dot{H}^{1/2}(\Rtwo))$, one argues as above. This takes care of the linear estimate \eqref{eq:LinearAbstractEstimate} for both Theorem \ref{thm:LWPNLW} and Remark \ref{rem:LWP}.

We turn to the trilinear estimate \eqref{eq:NonLinearAbstractEstimate}, as a consequence of Strichartz estimates. We will first prove for $0 < T \leq 1$
\begin{equation}
\label{eq:NonlinearEstimateI}
\begin{split}
&\quad \Big\| \int_0^t \frac{\sin((t-s) \sqrt{-\Delta})}{\sqrt{-\Delta}} (u_{1}\overline{u_{2}} u_{3})(s) ds \Big\|_{L_t^{24/7}([0,T],L^4(\R^2))} \\
&\lesssim T^{3/24} \prod_{i=1}^{3}\|u_{i}\|_{L_t^{24/7}([0,T],L^4(\R^2))}.
\end{split}
\end{equation}
To this end, for the high frequencies, we use Theorem \ref{thm:StrichartzEstimates} with $p_{1}=p_{2} = 8$, $q_1=q_{2} = 4$, to find
\begin{equation*}
\begin{split}
&\quad \Big\| \int_0^t \frac{\sin((t-s) \sqrt{-\Delta})}{\sqrt{-\Delta}} (1-\rho)(D) (u_{1}\overline{u_{2}}u_{3})(s) ds \Big\|_{L_t^{24/7}([0,T],L^4(\Rtwo))} \\
&\lesssim T^{1/6} \Big\| \frac{|D|^{6/8}}{|D|} (1-\rho)(D) (u_{1}\overline{u_{2}}u_{3}) \Big\|_{L_t^{8/7} L^{4/3}} \lesssim T^{1/8} \| u_{1}\overline{u_{2}}u_{3} \|_{L_t^{8/7} L^{4/3}}\\
&\lesssim T^{1/6} \prod_{i=1}^{3}\| u_{i} \|^3_{L_t^{24/7}([0,T], L^4(\Rtwo))}.
\end{split}
\end{equation*}
The low frequencies $\rho(D) (u_{1}\overline{u_{2}}u_{3})$ are estimated using Minkowski's inequality, Mikhlin's theorem and a Sobolev embedding:
\begin{equation*}
\begin{split}
&\quad \Big\| \int_0^t \frac{\sin((t-s) \sqrt{-\Delta})}{\sqrt{-\Delta}} \rho(D) (u_{1}\overline{u_{2}}u_{3})(s) ds \Big\|_{L_t^{24/7}([0,T],L^4(\Rtwo))} \\
&\lesssim T^{7/24} \| u_{1}\overline{u_{2}}u_{3}\|_{L_t^{1} W^{-1,4}} \lesssim T^{10/24} \| u_{1}\overline{u_{2}}u_{3} \|_{L_t^{8/7} L^{4/3}}\\
&\lesssim T^{10/24} \prod_{i=1}^{3}\| u_{i} \|^3_{L_t^{24/7}([0,T], L^4(\Rtwo))}.
\end{split}
\end{equation*}
This already concludes the proof for initial data in $W^{\veps,4}(\R^{2})\times W^{-1+\veps,4}(\Rtwo)$. 

For initial data involving the $\Ba^{s}_{4,2,2}(\Rtwo)$ spaces, we also need to consider the solution space $C([0,T],\dot{H}^{3/8}(\Rtwo))$. For the high frequencies, we use Minkowski's inequality and a Sobolev embedding:
\begin{equation*}
\begin{split}
&\; \Big\| \int_0^t \frac{\sin((t-s) \sqrt{-\Delta})}{\sqrt{-\Delta}} (1-\rho)(D) (u_1 \bar{u}_2 u_3)(s) ds \Big\|_{L^{\infty}_{t}([0,T],\dot{H}^{3/8}(\Rtwo))} \\
&\lesssim \| \langle D \rangle^{-5/8} (u_1 \overline{u_2} u_3) \|_{L_t^1([0,T], L^{2}(\Rtwo))} \lesssim T^{3/24} \prod_{i=1}^{3}\| u_i \|_{L_t^{24/7}([0,T], L^4(\Rtwo))}.
\end{split}
\end{equation*}
The argument for the low frequencies is almost identical, although one can use Plancherel's theorem to estimate away the singularity at zero:
\begin{equation*}
\begin{split}
&\; \Big\| \int_0^t \frac{\sin((t-s) \sqrt{-\Delta})}{\sqrt{-\Delta}} \rho(D) (u_1 \overline{u_2} u_3)(s) ds \Big\|_{L^{\infty}_{t}([0,T],\dot{H}^{3/8}(\Rtwo))} \\
&\lesssim \| u_1 \overline{u_2} u_3 \|_{L_t^1([0,T], L^{4/3}(\Rtwo))} \lesssim T^{3/24} \prod_{i=1}^3 \| u_i \|_{L_t^{24/7}([0,T], L^4(\Rtwo))}.
\end{split}
\end{equation*}
Since $T^{\frac{3}{24}} = \max( T^{\frac{10}{24}}, T^{\frac{3}{24}}, T^{\frac{1}{6}})$ for $0 < T \leq 1$, we choose this factor in \eqref{eq:NonlinearEstimateI}.
This proves the required supremum norm bounds, while the continuity statements are automatic, since the half-wave group is strongly continuous on $\dot{H}^{s}(\Rtwo)$. This also concludes the proof for initial data as in \eqref{eq:LWPNLW1}.

Finally, we deal with the trilinear estimate for $p=6$. We first prove for some $\kappa > 0$
\begin{equation}
\label{eq:NonlinearEstimateII}
\Big\| \int_0^t \frac{\sin((t-s) \sqrt{-\Delta})}{\sqrt{-\Delta}} (u_{1}\overline{u_{2}} u_{3})(s) ds \Big\|_{L_t^4([0,T],L^6(\Rtwo))}\lesssim T^{\kappa} \prod_{i=1}^{3}\|u_{i}\|_{L^{4}_{t}([0,T],L^{6}(\Rtwo))}.
\end{equation}
The estimate of the low frequencies is as before, so it suffices to use Strichartz estimates with $p_1 = q_1 = 6$ and $p_2 = \infty$, $q_2 = 2$:
\begin{equation*}
\begin{split}
&\quad \Big\| \int_0^t \frac{\sin((t-s) \sqrt{- \Delta})}{\sqrt{-\Delta}} (1-\rho)(D) (u_1 \overline{u_2} u_3)(s) ds \Big\|_{L_t^4([0,T], L^6(\Rtwo))} \\
&\lesssim T^{1/12} \Big\| \int_0^t \frac{\sin((t-s) \sqrt{- \Delta})}{\sqrt{-\Delta}} (1-\rho)(D) (u_1 \overline{u_2} u_3)(s) ds \Big\|_{L_t^6L^6_{x}} \\
&\lesssim T^{1/12} \| u_1 \overline{u_2} u_3 \|_{L_t^1 L^2_{x}} \lesssim T^{1/3} \prod_{i=1}^3 \| u_i \|_{L_t^4([0,T], L^6(\Rtwo))}.
\end{split}
\end{equation*}
This proves the required statement for initial data in $W^{1/6+\veps,6}(\R^{2})\times W^{-5/6+\veps,6}(\Rtwo)$ and concludes the proof of Remark \ref{rem:LWP}.

On the other hand, for the local well-posedness with initial data in $\mathcal{B}^{s}_{6,2,2}(\Rtwo)$ we also have to consider the solution space $C([0,T],\dot{H}^{1/2}(\Rtwo))$, in the following sense:
\begin{align*}
&\quad \Big\| \int_0^t \frac{\sin((t-s) \sqrt{- \Delta})}{\sqrt{- \Delta}} (u_1 \overline{u_2} u_3)(s) ds \Big\|_{L_t^\infty([0,T], \dot{H}^{1/2}(\Rtwo))}\\
&\lesssim T^{1/12} \prod_{i=1}^3 \| u_i \|_{L_t^4([0,T], L^6(\Rtwo))}.
\end{align*}
The estimate for the low frequencies is carried out by Plancherel's theorem, while for the high frequencies the argument is
\begin{equation*}
\begin{split}
&\quad\Big\| \int_0^t \frac{\sin((t-s) \sqrt{- \Delta})}{\sqrt{- \Delta}} (1-\rho(D)) (u_1 \bar{u}_2 u_3)(s) ds \Big\|_{L_t^\infty([0,T],\dot{H}^{1/2}(\Rtwo))} \\
&\lesssim \| u_1 \bar{u}_2 u_3 \|_{L_t^1 L_x^2}\lesssim T^{1/12} \prod_{i=1}^3 \| u_i \|_{L_t^4([0,T], L^6(\Rtwo))}.
\end{split}
\end{equation*}
Choosing $\kappa > 0$ such that $T^\kappa$ dominates the powers of $T$ obtained in the above estimates finishes the proof of \eqref{eq:NonlinearEstimateII}.
\end{proof}

We remark that there is slack in the spatial regularity in the nonlinear argument. This can be translated to solve the quintic nonlinear wave equation
\begin{equation}
\label{eq:QuinticNLW}
\begin{cases}
\partial_t^2 u - \Delta_{x} u = \pm |u|^4 u, \quad (t,x) \in \R \times \R^2, \\
u(0) = f \in \mathcal{B}^\varepsilon_{6,2,2}(\R^2) + \dot{H}^{1/2}(\R^2), \quad \dot{u}(0) = g \in \mathcal{B}^{\varepsilon-1}_{6,2,2}(\R^2) + \dot{H}^{-1/2}(\R^2)
\end{cases}
\end{equation}
in the solution space $S_T = L_t^6([0,T], L_x^6) \cap C([0,T], \mathcal{B}^\varepsilon_{6,2,2}(\R^2) + \dot{H}^{1/2}(\R^2))$ for small initial data. The crucial nonlinear estimate reads
\begin{equation*}
\begin{split}
&\; \Big\| \int_0^t \frac{\sin((t-s) \sqrt{- \Delta})}{\sqrt{-\Delta}} \prod_{i=1}^5 u_i(s) ds \Big\|_{L_t^6([0,T],L^6(\R^2))} \\
&\lesssim \Big\| \prod_{i=1}^5 u_i \Big\|_{L_t^{6/5}([0,T], L^{6/5}(\Rtwo))} \lesssim \prod_{i=1}^5 \| u_i \|_{L_t^6([0,T], L^6(\Rtwo))}
\end{split}
\end{equation*}
with Strichartz pairs $(p_i,q_i) = (6,6)$, $i=1,2$, because inhomogeneous Strichartz pairs as in Theorem \ref{thm:StrichartzEstimates} lose exactly one derivative. Note that we cannot afford to apply H\"older's inequality in time anymore. Hence, this argument does not allow to prove well-posedness for large initial data. This is not surprising because \eqref{eq:QuinticNLW} is $\dot{H}^{1/2}(\R^2) \times \dot{H}^{-1/2}(\R^2)$-scaling critical. Easy variants of the above arguments yield the following theorem.

\begin{theorem}
\label{thm:LWPQuinticNLW}
For any $T > 0$, there is an $\varepsilon > 0$ such that \eqref{eq:QuinticNLW} is analytically locally well posed with $u \in S_T = L^6([0,T],L^6(\R^2)) \cap C([0,T],\mathcal{B}^\varepsilon_{6,2,2} + \dot{H}^{1/2}(\R^2))$  provided that
\begin{equation*}
\| f \|_{\mathcal{B}^\varepsilon_{6,2,2}(\R^2) + \dot{H}^{1/2}(\R^2)} + \| g \|_{\mathcal{B}^{\varepsilon-1}_{6,2,2} + \dot{H}^{-1/2}} \leq \varepsilon.
\end{equation*}
\end{theorem}

\subsection{Results for slower decaying initial data}\label{subsec:slower}

In the following we point out how considering higher Picard iterates allows us to construct solutions for very slowly decaying initial data. The arguments are similar to \cite{Schippa2022} and \cite{DodsonSofferSpencer2021}, albeit with the difference that the Duhamel integral has a stronger smoothing effect. We consider the cubic nonlinear wave equation in $d$ dimensions:
\[
\left\{\!\begin{array}{cl}
\partial_t^2 u - \Delta_{x} u &= \pm |u|^2 u, \quad (t,x) \in \R \times \R^d, \; d \geq 2,\\
u(0) &= f_1 \in X, \quad \dot{u}(0) = f_2 \in Y,
\end{array} \right.
\]
although our main results concern $d\in\{2,3\}$. We write the solution abstractly:
\begin{equation*}
u= L(f_1,f_2) + N_3(u,u,u),
\end{equation*}
as in Section \ref{subsec:prelim}.

For $d,n\geq 2$, we consider initial data in $L^{4n+2}$-based spaces, and we let
\[
\begin{split}
u^0(t) &= L(f_1,f_2), \\
u^1(t) &= N_3(u^0,u^0,u^0),\\
u^j(t) &= N_3\Big( \sum_{k=0}^{j-1} u^k, \sum_{k=0}^{j-1} u^k, \sum_{k=0}^{j-1} u^k \Big) - \sum_{k=1}^{j-1} u^k, \quad (j \geq 2).
\end{split}
\]
We will prove the existence of a 
\[
v \in S^0([-1,1] \times \R^d) := L_t^\infty([-1,1];L^{2}_{x}(\R^d))\cap L_t^4([-1,1], L_x^\infty(\R^d))
\]
 which solves
\begin{equation*}
v = u - \sum_{j=0}^{n-1} u^j.
\end{equation*}
We can rewrite this as 
\begin{equation}
\label{eq:DifferenceSolution}
v = N_3(u,u,u) - \sum_{j=1}^{n-1} u^j = N_3\Big(v + \sum_{j=0}^{n-1} u^j, v+ \sum_{j=0}^{n-1} u^j, v + \sum_{j=0}^{n-1} u^j\Big)-\sum_{j=1}^{n-1} u^j
\end{equation}
for $j \geq 2$. One can check that $u^j$ contains only terms $A_k$ with $k \geq 2j+1$, where $A_{k}$ is as in \eqref{eq:Am} (cf. \cite{Schippa2022,DodsonSofferSpencer2021}). We therefore obtain estimates for such terms.

\begin{lemma}
\label{lem:IterationAm}
Let $d \in \{2, 3 \}$, $n \geq 2$ and $s > \frac{d-1}{2}( 1 - \frac{1}{4n+2})$. Then there exists a $C\geq0$ such that
\[
\| A_m(f_{1},f_{2}) \|_{L_t^\infty([-1,1];L^{\frac{4n+2}{m}}(\R^d)) } \leq C\big( \| f_{1} \|^m_{W^{s,4n+2}(\R^d)} + \| f_{2} \|^m_{W^{s-1,4n+2}(\R^d)}\big).
\]
for all $m \in \{1,\ldots, 2n-1\}$, $f_{1}\in W^{s,4n+2}(\Rd)$ and $f_{2}\in W^{s-1,4n+2}(\Rd)$.
\end{lemma}
\begin{proof}
First note that $A_{m}=0$ if $m$ is even. Hence we may suppose that $m = 2k+1$ for some $k\in\N_{0}$. Let $\veps>0$ and set $p:=4n+2$ and $q:=(4n+2)/m$. Then, by the embeddings \eqref{eq:EmbeddingLargep} and \eqref{eq:SobolevBesov6}, one has
\begin{align*}
\| A_m (f_1,f_2) \|_{L_t^\infty L^{q}_{x}} &\lesssim \| A_m (f_1,f_2) \|_{L_t^\infty \mathcal{B}^{s(q)}_{q,q',q'}} \lesssim \| A_m(f_1,f_2) \|_{L_t^\infty \mathcal{B}^{(d-1)/4}_{q,1,1}}\\
&\lesssim \| A_m(f_1,f_2) \|_{L_t^\infty \mathcal{B}^{s(1)+\veps}_{q,1,1}},
\end{align*}
since $s(1)=(d-1)/4$. We can use this regularity to iterate the Duhamel integral in adapted spaces, by Lemma \ref{lem:TrilinearEstimate} and because $d \in \{2,3\}$. 
First, we split the Duhamel integral into low and high frequencies:
\begin{equation*}
\begin{split}
&\quad \big\| \int_0^t \frac{e^{i(t-s) \sqrt{- \Delta}}}{\sqrt{- \Delta}} (u_1 u_2 u_3) (s) ds \big\|_{L_t^\infty \mathcal{B}^{s(1)+\varepsilon}_{q,1,1}} \\
&\leq \| \rho(D) \int_0^t \frac{e^{i(t-s) \sqrt{-\Delta}}}{\sqrt{- \Delta}} (u_1 u_2 u_3) (s) ds \big\|_{L_t^\infty L_x^q} \\
&\quad + \| (1-\rho(D)) \int_0^t e^{i(t-s) \sqrt{-\Delta}} (u_1 u_2 u_3)(s) ds \|_{L_t^\infty \mathcal{B}^{s(1)-1+\varepsilon}_{q,1,1}}.
\end{split}
\end{equation*}
The low frequencies are estimated by Mikhlin's theorem:
\begin{equation*}
\begin{split}
&\quad \| \rho(D) \int_0^t \frac{e^{i(t-s) \sqrt{-\Delta}}}{\sqrt{-\Delta}} (u_1 u_2 u_3)(s) ds \|_{L^q} \\
 &\lesssim T \| u_1 u_2 u_3 \|_{L_t^\infty L_x^q} \lesssim T \prod_{i=1}^3 \| u_i \|_{L_t^\infty L_x^{3q}} \lesssim T \prod_{i=1}^3 \| u_i \|_{L_t^\infty \mathcal{B}^{s(1)+\varepsilon}_{3q,1,1}},
\end{split}
\end{equation*}
which allows for iteration. Moreover, for the high frequencies we use the boundedness of $e^{it \sqrt{-\Delta}}$ on the adapted Besov spaces, and iterate the trilinear estimate in Lemma \ref{lem:TrilinearEstimate} $k$ times, to obtain
\begin{equation*}
\begin{split}
\| A_m (f_1,f_2) \|_{L_t^\infty \mathcal{B}^{s(1)+\veps}_{q,1,1}} &\lesssim \| f_1 \|^m_{\mathcal{B}^{s(1)+\veps}_{4n+2,1,1}} + \| f_2 \|^m_{\mathcal{B}^{s(1)+\veps-1}_{4n+2,1,1}} \\
&\lesssim \| f_1 \|^m_{\mathcal{B}^{s(1)+2\veps}_{p,p,p}} + \| f_2 \|^m_{\mathcal{B}^{s(1)+2\veps-1}_{p,p,p}} \\
&\lesssim \| f_1 \|^m_{W^{s(1)+s(p)+3\veps,p}} + \| f_2 \|^m_{W^{s(1)+s(p)-1+3\veps,p}}.
\end{split}
\end{equation*}
Here we also used the embeddings \eqref{eq:SobolevBesov4}, \eqref{eq:SobolevBesov7}, \eqref{eq:SobolevBesov5} and \eqref{eq:Sobolevintro}.
By choosing $\veps$ sufficiently small, this concludes the proof.
\end{proof}

Similarly, we can iterate
\begin{equation*}
\| A_m(f_1,f_2) \|_{L_t^\infty L_x^\infty} \lesssim \| A_m (f_1,f_2) \|_{L_t^\infty \mathcal{B}^{\frac{d-1}{4}+\varepsilon}_{\infty,1,1}} \lesssim \| f_1 \|^m_{\mathcal{B}^{\frac{d-1}{4}+\varepsilon}_{\infty,1,1}} + \|f_2 \|^m_{\mathcal{B}^{\frac{d-1}{4}-1+\varepsilon}_{\infty,1,1}}
\end{equation*}
and
\begin{equation*}
\| f \|_{\mathcal{B}^{\frac{d-1}{4}+\varepsilon}_{\infty,1,1}} \lesssim \| f \|_{\mathcal{B}^{\frac{d-1}{4}+\frac{d+1}{2p}+\varepsilon}_{p,1,1}} \lesssim \| f \|_{\mathcal{B}^{\frac{d-1}{4}+\frac{d+1}{2p}+\varepsilon}_{p,p,p}} \lesssim \| f \|_{W^{s,p}}
\end{equation*}
for $s > \alpha$ with
\begin{equation}
\label{eq:RegularityLpSpaces}
\alpha := \frac{d-1}{4}+\frac{d+1}{2p} + \frac{d-1}{2} \big( \frac{1}{2} - \frac{1}{p} \big) = \frac{d-1}{2} \big( 1 - \frac{1}{p} \big) + \frac{d+1}{2p}.
\end{equation}
This shows that
\begin{equation*}
\| A_m(f_1,f_2) \|_{L_t^\infty L_x^\infty} \lesssim \| f_1 \|^m_{W^{s,p}} + \| f_2 \|^m_{W^{s-1,p}}.
\end{equation*}
We can argue like above to find
\begin{equation*}
\| A_m(f_1,f_2) \|_{L_t^\infty L_x^\infty} \lesssim \| f_1 \|^m_{\mathcal{B}^{\frac{d-1}{4}+\varepsilon}_{\infty,1,1}} + \|f_2 \|^m_{\mathcal{B}^{\frac{d-1}{4}-1+\varepsilon}_{\infty,1,1}}.
\end{equation*}
Now we use the embeddings
\begin{equation*}
\| f \|_{\mathcal{B}^{s}_{\infty,1,1}} \lesssim \| f \|_{\mathcal{B}^{s+\frac{d+1}{2p}}_{p,1,1}} \lesssim \| f \|_{\mathcal{B}^{s+\frac{d+1}{2p}}_{p,2,2}}.
\end{equation*}
This shows that 
\begin{equation*}
\| A_m(f_1,f_2) \|_{L_t^\infty L_x^\infty} \lesssim \| f_1 \|^m_{\mathcal{B}^{s}_{p,2,2}} + \| f_2 \|^m_{\mathcal{B}^{s-1}_{p,2,2}}
\end{equation*}
for $s> \tilde{\alpha}$ with
\begin{equation}
\label{eq:RegularityAdaptedSpaces}
\tilde{\alpha} = \frac{d-1}{4} + \frac{d+1}{2p}.
\end{equation}
We have proved the following lemma, regarding the $u^{j}$ from above.
\begin{lemma}
Let $d \in \{ 2, 3 \}$, $n \geq 2$, $0 \leq j \leq n-1$, and $p = 4n+2$. Let $\alpha$ and $\tilde{\alpha}$ be given by \eqref{eq:RegularityLpSpaces} and \eqref{eq:RegularityAdaptedSpaces}. Then, for $\varepsilon >0$, there are $\varepsilon_n \leq 1$ and $\tilde{\varepsilon}_n \leq 1$ such that
\begin{equation*}
\| u^j \|_{L^\infty_{t,x}([0,1] \times \R^d)} + \| u^j \|_{L_t^\infty L^{\frac{4n+2}{2j+1}}([0,1] \times \R^d)} \lesssim \| f \|_{W^{\alpha+\veps,4n+2}(\R^d)}
\end{equation*}
holds true provided that $\| f \|_{W^{\alpha+\veps,4n+2}(\Rd)} \leq \varepsilon_n$, and
\begin{equation*}
\| u^j \|_{L^\infty_{t,x}([0,1] \times \R^d)} + \| u^j \|_{L^\infty_t L^{\frac{4n+2}{2j+1}}([0,1] \times \R^d)} \lesssim \| f \|_{\mathcal{B}^{\tilde{\alpha}+\varepsilon}_{4n+2,2,2}(\R^d)}
\end{equation*}
provided that $\| f \|_{\mathcal{B}^{\tilde{\alpha}+\varepsilon}_{4n+2,2,2}(\R^d)} \leq \tilde{\varepsilon}_n$.
\end{lemma}
With the estimate for the higher Picard iterates at hand, the following proposition is proved like in \cite[Proposition~4.6]{Schippa2022}.

\begin{proposition}
\label{prop:ExistenceDifferenceSolution}
Let $d \in \{2,3\}$, $\varepsilon > 0$, $n \geq 2$, and $\varepsilon_n$, $\tilde{\varepsilon}_n \leq 1$ be as in Lemma \ref{lem:IterationAm}. Then, there is a unique $v \in S^0$, which solves \eqref{eq:DifferenceSolution} with $v(0) = \dot{v}(0) = 0$.
\end{proposition}
This yields the following theorem on local well-posedness for slowly decaying initial data. We focus on the two-dimensional case with small data to simplify the Strichartz space, but there are clearly analogs available in higher dimensions.
\begin{theorem}
\label{thm:SlowlyDecayingData}
Let $d=2$, $\varepsilon > 0$, $n \geq 2$, and $(f_1,f_2)$, $\varepsilon_n$, and $\tilde{\varepsilon}_n$ like in Proposition \ref{prop:ExistenceDifferenceSolution}. Let $\frac{2}{p} + \frac{1}{4n+2} = \frac{1}{2}$. Then, there is $u \in L_t^p([0,1],L^{4n+2}(\R^2))$ which solves \eqref{eq:CubicNLW}. Furthermore, for 
\begin{align*}
\| (f_1,f_2) \|_{W^{\alpha+\veps,4n+2} \times W^{\alpha+\veps-1,4n+2}} + \| (g_1,g_2) \|_{W^{\alpha+\veps,4n+2} \times W^{\alpha+\veps-1,4n+2}} &\leq \varepsilon_n \\
\text{ or } 
\| (f_1,f_2) \|_{\mathcal{B}^{\tilde{\alpha}+\varepsilon}_{4n+2,2,2} \times \mathcal{B}^{\tilde{\alpha}+\varepsilon}_{4n+2,2,2}} + \| (g_1,g_2) \|_{\mathcal{B}^{\tilde{\alpha}-1+\varepsilon}_{4n+2,2,2} \times \mathcal{B}^{\tilde{\alpha}-1+\varepsilon}_{4n+2,2,2}} &\leq \tilde{\varepsilon}_n,
\end{align*}
the corresponding solutions satisfy $\| u_1 - u_2 \|_{L^p([0,1],L^{4n+2})} \to 0$ provided that the initial data converges in the spaces of initial data.
\end{theorem}

\subsection{Global well-posedness results}
\label{subsection:GlobalNLW}
We prove global results for the defocusing cubic nonlinear wave equation in two dimensions:
\begin{equation}
\label{eq:DefocusingWaveEquation}
\left\{ \begin{array}{cl}
\partial_t^2 u - \Delta_{x} u &= -|u|^2 u, \quad (t,x) \in \R \times \R^2, \\
u(0) &= f_1 \in \mathcal{B}^{s}_{p,2}(\R^2), \qquad \dot{u}(0) = f_2 \in \mathcal{B}^{s-1}_{p,2}(\R^2).
\end{array} \right.
\end{equation}
We focus on the case where $p=6$.

The main result of this section is the following:
\begin{theorem}
\label{thm:GlobalNLW}
Let $s > 1/2$, and $(f_1,f_2) \in \mathcal{B}^{s}_{6,2,2}(\R^2) \times \mathcal{B}^{s-1}_{6,2,2}(\R^2)$. Then, for any $T>0$, there is a global solution $u \in L_t^4([0,T],L^6(\R^2))$ to \eqref{eq:DefocusingWaveEquation}.
\end{theorem}

In the following the arguments from \cite{Schippa2022} are adapted, which were previously applied to nonlinear Schr\"odinger equations. To avoid technicalities, we shall consider Schwartz initial data which admit global solutions and allow for integration by parts arguments. The a priori assumption can be removed later by well-posedness and limiting arguments.
We denote the linear part of the solution to \eqref{eq:DefocusingWaveEquation} by
\begin{equation*}
w(t) = \cos(t \sqrt{- \Delta}) f_1 + \frac{\sin (t \sqrt{- \Delta})}{\sqrt{- \Delta}} f_2.
\end{equation*}
The difference with the full solution is given by
\begin{equation}
\label{eq:Difference}
v(t) = u(t) - w(t) = - \int_0^t \frac{\sin((t-s) \sqrt{-\Delta})}{\sqrt{- \Delta}} (|v+w|^2 (v+w)) ds.
\end{equation}
We have the following blow-up alternative (cf. \cite[Lemma~4.9]{Schippa2022}).

\begin{lemma}
\label{lem:BlowUpAlternative}
Let $s > 0$, $(f_1,f_2) \in (\mathcal{B}^{s}_{6,2,2}(\R^2) + H^{1}(\R^2)) \times (\mathcal{B}^{s-1}_{6,2,2}(\R^2) + L^2(\R^2))$, and $u$ be the solution to \eqref{eq:DefocusingWaveEquation} provided by Theorem \ref{thm:LWPNLW} in $L_t^4([0,T],L^6(\R^2))$. If $T^*$ is maximal such that $u \in L_t^4([0,T],L^6(\R^2))$ for $T < T^*$, but $u \notin L_t^4([0,T^{*}],L^6(\R^2))$, then $\lim_{t \to T^*} ( \| v(t) \|_{H^{1}(\R^2)} + \| \partial_t v(t) \|_{L^2} ) = \infty$ with $v$ defined like in \eqref{eq:Difference}.
\end{lemma}
\begin{proof}
We note that for the free solution we have
\begin{align}
\label{eq:FreeEstimateI}
\sup_{t \in [0,T]} \| w(t) \|_{\mathcal{B}^s_{6,2,2}+H^1} &\lesssim_T \| (f_{1},f_{2}) \|_{(\mathcal{B}^{s}_{6,2,2} + H^1) \times (\mathcal{B}^{s-1}_{6,2,2} + L^2)}, \\
\label{eq:FreeEstimateII}
\sup_{t \in [0,T]} \| \partial_t w(t) \|_{\mathcal{B}^{s-1}_{6,2,2}+L^2} &\lesssim_T \| (f_{1},f_{2}) \|_{(\mathcal{B}^{s}_{6,2,2} + H^1) \times (\mathcal{B}^{s-1}_{6,2,2} + L^2)}.
\end{align}
We further argue by contradiction. Suppose that there is a sequence $(t_n) \subseteq [0,T^*)$ with $t_n \uparrow T^*$ and
\begin{equation}
\label{eq:InhomogeneousEstimate}
\lim_{n \to \infty} ( \| v(t_n) \|_{H^1} + \| v(t_n) \|_{L^2} ) \leq C.
\end{equation}
But by Theorem \ref{thm:LWPNLW} and Remark \ref{rem:LWP}, we can solve the nonlinear wave equation with initial data
\begin{equation*}
w(t_n) + v(t_n) \in \mathcal{B}^s_{6,2,2} + H^1, \quad \dot{w}(t_n) + \dot{v}(t_n) \in \mathcal{B}^{s-1}_{6,2,2} + L^2
\end{equation*}
for times $T=T(\| w(t_n) + v(t_n) \|_{\mathcal{B}^s_{6,2,2} + H^1}, \| \dot{w}(t_n) + \dot{v}(t_n) \|_{\mathcal{B}^{s-1}_{6,2,2} + L^2})$ and by \eqref{eq:FreeEstimateI}, \eqref{eq:FreeEstimateII}, and \eqref{eq:InhomogeneousEstimate}, we find
\begin{equation*}
\begin{split}
\;&\| w(t_n) + v(t_n) \|_{\mathcal{B}^s_{6,2,2} + H^1} + \| \dot{w}(t_n) + \dot{v}(t_n) \|_{\mathcal{B}^{s-1}_{6,2,2} + L^2}\\
& \lesssim_{T^*} C + \| (f_{1},f_{2}) \|_{(\mathcal{B}^{s}_{6,2,2} + H^1) \times (\mathcal{B}^{s-1}_{6,2,2} + L^2)}.
\end{split}
\end{equation*}
This means the local existence time is bounded from below, which yields a contradiction because it means we can continue the solution beyond $T^*$. Hence, the solution is global.
\end{proof}

Hence, for the proof of global well-posedness it suffices to show
\begin{equation*}
\sup_{t \in [0,T]} (\| v(t) \|_{H^{1}(\R^2)} + \| \partial_t v(t) \|_{L^2(\Rtwo)}) \leq C(T).
\end{equation*}
Recall that mass and energy are conserved quantities for (smooth) solutions to \eqref{eq:DefocusingWaveEquation}:
\begin{align}
\label{eq:Mass}
M(u) &= \int_{\R^2} |u|^2 dx, \\
\label{eq:Energy}
E(u) &= \int_{\R^2} \frac{1}{2} |\partial_t u|^2 + \frac{1}{2} |\nabla_x u|^2 + \frac{1}{4} |u|^4 dx.
\end{align}
But the quantities are not conserved for differences of solutions or $v$. Still we can control $M(v)+E(v)$ by Gr\o nwall's argument for sufficiently regular initial data like in \cite{DodsonSofferSpencer2021,Schippa2022} in the context of the defocusing nonlinear Schr\"odinger equation.
In the proof we have to control $\| w(t) \|_{L^6(\R^2)}$ and $\| w(t) \|_{L^\infty(\Rd)}$, for which we use embeddings, namely \eqref{eq:EmbeddingLargep}, Propositions \ref{prop:Sobolev} and \ref{prop:Sobolev2}, and a standard Sobolev embedding:
\begin{align*}
\| w(t) \|_{L^6(\R^2)} &\lesssim \| w(t) \|_{\mathcal{B}^{1/6}_{6,6,6}(\Rtwo)} \lesssim \| w(t) \|_{\mathcal{B}^{1/6+\veps}_{6,2,2}(\R^2)}, \\
\| w(t) \|_{L^\infty(\R^2)} &\lesssim \|w(t) \|_{W^{1/3+\veps,6}(\R^2)} \lesssim \| w(t) \|_{\mathcal{B}^{1/2+2\varepsilon}_{6,2,2}(\R^2)}.
\end{align*}
We show the following.

\begin{proposition}
\label{prop:GrowthBoundMassEnergy}
Let $\varepsilon > 0$, $(f_1,f_2) \in \mathcal{B}^{1/2+\varepsilon}_{6,2,2}(\Rtwo) \times \mathcal{B}^{-1/2+\varepsilon}_{6,2,2}(\Rtwo)$, and $T>0$. With notation as above, the following estimate holds for $0 \leq t \leq T$:
\begin{equation*}
\partial_t (M(v) + E(v) + 1)(t) \leq C_T ( M(v) + E(v) + 1)
\end{equation*}
for all $0 \leq t \leq T$.
\end{proposition}
With Proposition \ref{prop:GrowthBoundMassEnergy} in place, we find by Gr\o nwall's argument:
\begin{equation*}
(M(v) + E(v) + 1)(t) \leq e^{\int_0^t C(s) ds},
\end{equation*}
and hence, $M(v)$ and $E(v)$ do not blow up. Theorem \ref{thm:GlobalNLW} follows.
\begin{proof}[Proof~of~Proposition~\ref{prop:GrowthBoundMassEnergy}]
We introduce the notation
\begin{equation*}
(f,g) = \Re \int_{\R^2} f(x) \overline{g}(x) dx.
\end{equation*}
For the growth of $M(v)$, we find
\begin{equation*}
\partial_t M(v) = 2(\partial_t v,v) \lesssim E(v)^{1/2} M(v)^{1/2} \lesssim M(v) + E(v) + 1.
\end{equation*}
For the time-derivative of $E$ we find
\begin{equation*}
\begin{split}
\partial_t E(v) &= ((\partial_t^2 v), \partial_t v) + (\partial_t \nabla_x v, \nabla_x v) + (\partial_t v, |v|^2 v) \\
&= (\partial_t v, \partial_t^2 v - \Delta v + |v|^2 v) \\
&= (\partial_t v, - |v+w|^2 (v+w) + |v|^2 v) \\
&\lesssim |(\partial_t v, |v|^2 w)| + |(\partial_t v, v |w|^2)| + |(\partial_t v, |w|^2 w)| \\
&\lesssim \| \partial_t v \|_{L^2} \| v \|_{L^4}^{2} \| w \|_{L^\infty} + \| \partial_t v \|_{L^2} \| v \|_{L^2} \| w \|_{L^\infty}^2 + \| \partial_t v \|_{L^2} \| w \|_{L^6}^3 \\
&\lesssim_T E(v) + E(v)^{1/2} M(v)^{1/2} + E(v)^{1/2}.
\end{split}
\end{equation*}
This finishes the proof.
\end{proof}
We sketch the extension to slower decaying initial data:
\begin{theorem}
\label{thm:GWPSlowerDecay}
Let $d=2$, $\varepsilon >0$, $n \geq 2$, $\tilde{\alpha}$ like in \eqref{eq:RegularityAdaptedSpaces}. Let $(f_1,f_2) \in \mathcal{B}^{\tilde{\alpha}+\epsilon}_{4n+2,2,2}(\Rtwo) \times \mathcal{B}^{\tilde{\alpha}+\varepsilon - 1}_{4n+2,2,2}(\Rtwo)$. Let $\frac{2}{p} + \frac{1}{4n+2} = \frac{1}{2}$. Then, there is some $\tilde{p} < p$ such that for any $T > 0$ there is $u \in L_t^{\tilde{p}}([0,T],L^{4n+2}(\R^2))$, which solves \eqref{eq:DefocusingWaveEquation}.
\end{theorem}
The key point is that solving \eqref{eq:DefocusingWaveEquation} in $L_t^p L_x^q$-spaces provides us with a blow-up alternative:

\begin{lemma}
\label{lem:BlowUpSlowerDecayingData}
Let $s > \tilde{\alpha}$, 
\[
(f_1,f_2) \in (\mathcal{B}^s_{4n+2,2,2}(\R^2) + H^1(\R^2),\mathcal{B}^{s-1}_{4n+2,2,2}(\R^2) + L^2(\R^2)),
\]
 and let $u$ be the solution to \eqref{eq:DefocusingWaveEquation} provided by Theorem \ref{thm:SlowlyDecayingData} in $L_t^{\tilde{p}}([0,T],L^{4n+2}(\R^2))$. If $T^*$ is maximal such that $u \in L_t^{\tilde{p}}([0,T],L^{4n+2}(\R^2))$ for $T < T^*$, but we have $u \notin L_t^{\tilde{p}}([0,T],L^{4n+2}(\R^2))$, then $\lim_{t \to T^*}(\| v(t) \|_{H^1} + \| \partial_t v(t) \|_{L^2}) = \infty$ with $v$ defined like in \eqref{eq:DifferenceSolution}.
\end{lemma}
The proof of Lemma \ref{lem:BlowUpSlowerDecayingData} follows along the lines of the proof of Lemma \ref{lem:BlowUpAlternative}. We turn to the proof of Theorem \ref{thm:GWPSlowerDecay}.
\begin{proof}[Proof~of~Theorem~\ref{thm:GWPSlowerDecay}]
By Lemma \ref{lem:BlowUpSlowerDecayingData}, for the proof of Theorem \ref{thm:GWPSlowerDecay} it suffices to show
\begin{equation*}
E(v) + M(v) + 1 \lesssim_T 1.
\end{equation*}
We use again Gr\o nwall's argument: We have like above
\begin{equation*}
\partial_t M(v) \lesssim M(v)^{1/2} E(v)^{1/2},
\end{equation*}
and we compute with $u_n = \sum_{j=0}^{n-1} u^j$
\begin{equation*}
\begin{split}
\partial_t E(v) &= (\partial_t v, \partial_t^2 v) + (\partial_t \nabla_x v, \nabla_x v) + (\partial_t v, |v|^2 v) \\
&= (\partial_t v, \partial_t^2 v - \Delta v + |v|^3) \\
&= (\partial_t v, -|v+u_n|^2 (v+u_n) + |v|^3 + \big| \sum_{j=0}^{n-2} u^j \big|^2 \sum_{j=0}^{n-2} u^j).
\end{split}
\end{equation*}
We can write schematically
\begin{equation*}
\begin{split}
&\quad -|v+u_n|^2 (v+u_n) + |v|^3 + \big| \sum_{j=0}^{n-2} u^j \big|^2 \sum_{j=0}^{n-2} u^j \\
&= A_{(2,1)}(v,u_n) + A_{(1,2)}(v,u_n) + \big[ \big| \sum_{j=0}^{n-2} u^j \big|^2 \sum_{j=0}^{n-2} u^j - \big| \sum_{j=0}^{n-1} u^j \big|^2 \sum_{j=0}^{n-1} u^j \big]
\end{split}
\end{equation*}
with $A_{(i,j)}(f,g)$ denoting terms which are homogeneous of degree $i$ in $f$ and of degree $j$ in $g$. We can estimate
$\| u_n(t) \|_{L_x^\infty} \lesssim_t 1$:
\begin{equation*}
|(\partial_t v, A_{(2,1)}(v,u_n) )| \lesssim \| \partial_t  v\|_{L^2} \|v \|^2_{L^4} \| u_n \|_{L^\infty} \lesssim_T E(v),
\end{equation*}
and
\begin{equation*}
|(\partial_t v, A_{(1,2)}(v,u_n))| \lesssim \| \partial_t v \|_{L^2} \| v \|_{L^2} \| u_n \|^2_{L^\infty} \lesssim_T E(v)^{1/2} + M(v)^{1/2}.
\end{equation*}
At last, we rewrite
\begin{equation*}
\big| \sum_{j=0}^{n-2} u^j \big|^2 \sum_{j=0}^{n-2} u^j - \big| \sum_{j=0}^{n-1} u^j \big|^2 \sum_{j=0}^{n-1} u^j = - \sum_{k,m} u^{n-1} u^k u^m
\end{equation*}
up to complex conjugates on the right-hand side. We estimate by H\"older's inequality
\begin{equation*}
|(\partial_t v, \sum_{k,m} u^{n-1} u^k u^m)| \lesssim \sum_{k,m} \| \partial_t v \|_{L^2} \| u^{n-1} \|_{L^{\frac{4n+2}{2n-1}}} \| u^k \|_{L^{4n+2}} \| u^m \|_{L^{4n+2}} \lesssim_T E(v)^{1/2}
\end{equation*}
noting that $u^k \in L^{4n+2}$ for any $k \geq 0$. This shows
\begin{equation*}
\partial_t (E(v) + M(v) + 1) \leq C(T) (E(v) + M(v) + 1),
\end{equation*}
and the proof is complete.
\end{proof}

\section*{Acknowledgements}

The first author would like to thank Po-Lam Yung for various helpful discussions, and the authors are grateful to the referee for their careful reading of the manuscript and for useful comments.

The research leading to these results has received funding from the Norwegian Financial Mechanism 2014-2021, grant 2020/37/K/ST1/02765. R.S. is supported by the Deutsche Forschungsgemeinschaft (DFG, German Research Foundation) -- Project-ID 258734477 -- SFB 1173.

\end{document}